\documentclass[preprint,3p]{elsarticle}
\usepackage{slashbox}
\usepackage{amsmath,amsopn,amssymb,epsfig,amsthm}
\usepackage{multirow}
\usepackage{color}

\definecolor{amethyst}{rgb}{0.6, 0.4, 0.8}
\definecolor{orange}{rgb}{1,0.5,0}

\newtheorem{Theorem}{Theorem}[section]
\newtheorem{Lemma}{Lemma}[section]
\newtheorem{Proposition}{Proposition}[section]
\newtheorem{Remark}{Remark}[section]

\newtheorem{example}{Example}[section]
\newtheorem{Definition}{Definition}[section]
\newtheorem{Corollary}{Corollary}[section]
\newproof{pot}{Proof of Theorem \ref{thm2}}
\newcommand{\bb}{\begin{bmatrix}}
\newcommand{\eb}{\end{bmatrix}}
\newcommand{\bl}[1]{\begin{list}{#1}{\usecounter{bean}}} \newcommand{\el}{\end{list}}
\newcommand{\bel}[1]{\begin{equation} \label{#1}} \newcommand{\eel}{\end{equation}}

\def\r2n2n{\mathbb{R}^{2n\times 2n}}
\def\c2n2n{\mathbb{C}^{2n\times 2n}}

\begin{document}

\date{}
\begin{frontmatter}
\title{On the existence of extremal solutions for the conjugate discrete-time Riccati equation}
%

\author{Chun-Yueh Chiang\corref{cor2}}
\ead{chiang@nfu.edu.tw}
\address{Center for Fundamental Sciences, National Formosa
University, Huwei 632, Taiwan.}

\cortext[cor2]{Corresponding author}

\date{ }

\begin{abstract}
In this paper we consider a class of conjugate discrete-time Riccati equations (CDARE), arising originally from the linear quadratic regulation problem for discrete-time antilinear systems. Recently, we have proved the existence of
the maximal solution to the CDARE with a nonsingular control weighting matrix under the framework of the constructive method.
Our contribution in the work is to finding another meaningful Hermitian solutions, which has received little attention in this topic.
Moreover, we show that some extremal solutions cannot be attained at the same time, and almost (anti-)stabilizing solutions coincide with some extremal solutions.
It is to be expected that our theoretical results presented in this paper will
play an important role in the optimal control problems for discrete-time antilinear systems.
\end{abstract}


\end{frontmatter}

\section{Introduction} \label{sec1}
%
A conjugate
discrete-time algebraic Riccati equation (CDARE) is the matrix equation
\begin{subequations}\label{cdare}
\begin{align}
  X = \mathcal{R}(X):= A^H \overline{X}A - A^H \overline{X}B R_X^{-1} B^H \overline{X}A + H, \label{cdare-a}
\end{align}
{\color{black} or its equivalent expression}
 \begin{align}
 X = A^H \overline{X} (I+ G\overline{X})^{-1} A + H, \label{cdare-b}
\end{align}
\end{subequations}
where $A\in \mathbb{C}^{n\times n}$, $B\in \mathbb{C}^{n\times m}$, $R\in \mathbb{H}_m$ is nonsingular,
$H=CC^H\in \mathbb{H}_n$ with $C\in \mathbb{C}^{n\times n}$, $I$ is the identity matrix of compatible size,
$G:= BR^{-1}B^H$, $R_X:=R+B^H \overline{X}B$ and the unknown $X\in\mathrm{dom}(\mathcal{R})$, respectively. Here, $\mathbb{H}_\ell$ denotes the set of all $\ell \times \ell$ Hermitian matrices and $\mathrm{dom}(\mathcal{R}) := \{X\in \mathbb{H}_n\, |\, \det(R_X)\neq 0\}$. We also notice that
$R_X$ is nonsingular if and only if $\det(I+G\overline{X})\neq 0$.

A class of CDAREs \eqref{cdare} arises from the linear quadratic regulation (LQR) optimal control problem for the discrete-time antilinear system
of the state space representation
\begin{equation} \label{antils}
  x_{k+1} = A\overline{x}_k + B\overline{u}_k,\quad k\geq 0,
\end{equation}
where $x_k\in \mathbb{C}^n$ is the state response and $u_k\in \mathbb{C}^m$ is the control input.
 The main goal of this control problem is to find a state feedback control $u_k = -Fx_k$ such that the performance index
 \[
 \mathcal{J}(u_k, x_0) := \sum_{k=0}^\infty \bb x_k^H & u_k^H\eb
 \bb H & \\ & R\eb\bb x_k \\ u_k\eb
 \]
is minimized with $H\geq 0$ and $R > 0$. If the antilinear system \eqref{antils} is controllable~\cite[Theorem 12.7]{w.z17} or the matrix pair $(A\overline{A},\bb B & A\overline{B}\eb)$ is stabilizable~\cite[Theorem 2]{MR4086472}, the optimal state feedback controller is $u_k^* := -R_{X_*}^{-1}B^H\mathbf{\overline{Z}_M}A x_k$ for $k\geq 0$ and thus
the minimum value of $\mathcal{J}(u_k^*, x_0) = x_0^H \mathbf{Z_M} x_0$ is achieved,
where $\mathbf{Z_M}:=\max(\mathbb{R}_=\cap\mathbb{N}_n)$ is the unique positive solution of the CDARE \eqref{cdare-a}. Here, the notation $\mathbb{N}_n$ collects all positive semidefinite matrices with size $n$ and all Hermitian solution of CDARE~\eqref{cdare} is denoted by $\mathbb{R}_=$.
 The positive definite solution of CDARE~\eqref{cdare} has been an active area of research and CDARE~\eqref{cdare} is also called the {\em discrete-time algebraic anti-Riccati equation} \cite{Wu.16,w.z17}. For more details see~\cite{w.z17,MR4086472}.

In the past few years, the authors~\cite{CHIANG202471,MR4487984} discovered a way to construct the existence of the maximal solution $\mathbf{X_M}:=\max(\mathbb{R}_=\cap\mathbb{P})$ for the CDARE~\eqref{cdare} with $G,H\in\mathbb{H}_n$. Here, the notation $\mathbb{P}$ collects any Hermitian matrix $X$ such that $R_X>0$. More precisely,
a class of the fixed-point iteration (FPI) has been utilized to compute $\mathbf{X_M}$ under the weaker assumptions. Further exploitation in the convergence result of the FPI appeared in recent work~\cite{MR4487984} or see the below Theorem~\ref{thm3p1}. It is clear that $\mathbf{X_M}\geq\mathbf{Z_M}$ when $G,H\geq 0$ since $\mathbb{N}_n\subseteq\mathbb{P}$. However, $\mathbf{X_M}-\mathbf{Z_M}$ may be indefinite in general if they exist. One aim of this work is to establish some appropriate conditions for the existence and relation to all meaningful solutions including $\mathbf{X_M}$ and $\mathbf{Z_M}$.

Based on these recent works, it is natural to ask whether another meaningful extremal solution $\mathbf{X_m}:=\max(\mathbb{R}_=\cap\mathbb{P})$, $\mathbf{Y_M}:=\max(\mathbb{R}_=\cap\mathbb{E})$ and $\mathbf{Y_m}:=\min(\mathbb{R}_=\cap\mathbb{E})$ exist. Here, the notation $\mathbb{E}$ collects any Hermitian matrix $X$ such that $R_X<0$.
In analogy with some particular Hermitian solutions of discrete-time algebraic Riccati equation(DARE)~\cite{MR4794239,CHIANG202471}, another interesting issue including almost stabilizing, almost anti-stabilizing, positive semidefinite and negative semidefinite solutions will also be defined and investigated. For the sake of explanation, the maximal element of $\mathbb{R}_{=}\cap\mathbb{N}_n$, the minimal element of $\mathbb{R}_{=}\cap\mathbb{N}_n$, the maximal element of $\mathbb{R}_{=}\cap -\mathbb{N}_n$, the minimal element of $\mathbb{R}_{=}\cap-\mathbb{N}_n$, the maximal element of $\mathbb{R}_{=}\cap\mathbb{E}$ and the minimal element of $\mathbb{R}_{=}\cap\mathbb{E}$ are denoted by $\mathbf{\mathbf{Z_M}}$, $\mathbf{Z_m}$, $\mathbf{\mathbf{W_M}}$, $\mathbf{W_m}$, $\mathbf{X_{a.s.}}$ and $\mathbf{X_{a.a.s.}}$, respectively, if they exist. For more details about the specific notations in the paper, see Tables~\ref{Fx5} and \ref{Fx6}.

As expected, we shall provide a constructive mathod to guarantee the existence of all extremal and almost (anti-)stabilizing solutions under reasonable assumptions, which generalizes the previously works. Two kinds of fixed-point iterations are proposed for finding all extremal solutions via the
two auxiliary algebraic Riccati equations.
Precisely, starting with $X_0\in \mathbb{H}_n$,
the sequence $\{X_k\}_{k=0}^\infty$ generated by the following FPI of the form
\begin{equation} \label{fpi}
X_{k+1} = \mathcal{R}(X_k), \quad k\geq 0,
\end{equation}
will be used in the following content of the paper.
With the help of the concept of FPI~\eqref{fpi},
it is shown in Theorem 3.1 of \cite{MR4487984} that the existence of the maximal element of $\mathbb{R}_\leq \cap \mathbb{P}$ to the CDARE \eqref{cdare} can be constructed iteratively by FPI~\eqref{fpi}, which is exactly equal to $\mathbf{X_M}=\max(\mathbb{R}_=\cap\mathbb{P})$.
Analogous to Theorem 3.1 of \cite{MR4487984}, one can easily obtain the minimal element of $\mathbb{R}_\geq \cap \mathbb{E}$ to the CDARE \eqref{cdare-a} based on the framework of the FPI~\eqref{fpi}. Later in Section~2, we will describe in more detail.


It is interesting to ask whether the minimal element $\mathbf{X_m}$ of $\mathbb{R}_= \cap \mathbb{P}$ or the maximal element $\mathbf{Y_M}$ of $\mathbb{R}_= \cap \mathbb{E}$
can be established in a unified framework of the fixed-point iteration~\eqref{fpi}? The following example gives a counterexample to this question. Namely, there is a possibility that the CDARE~\eqref{cdare} may has a minimal solution $\mathbf{X_m}$ for which $\mathbf{X_m}=\min(\mathbb{R}_=\cap\mathbb{P})$. However, the FPI \eqref{fpi} with any initial $X_0$ doesn't converge to $\mathbf{X_m}$.
\begin{example} \label{ex1}\em
We consider a scalar CDARE \eqref{cdare} with $n=m=1$ of the form
\begin{equation*} 
  x  = |a|^2 \bar{x} -\frac{|a|^2 \bar{x}^2 |b|^2}{r + |b|^2\bar{x}} + h 
   = \frac{|a|^2 \bar{x}}{1+g\bar{x}} + h, 
\end{equation*}
where $a=e^{\frac{\pi}{6}i}$, $b=e^{\frac{\pi}{3}i}$, $r=h=1$ and thus $g= |b|^2 / r=1$. Therefore, it has exactly two Hermitian solutions $\mathbf{x_M}=\frac{1+\sqrt{5}}{2}> \frac{1-\sqrt{5}}{2}=\mathbf{x_m}$. Note that $\mathbf{x_m}$ and $\mathbf{x_M}$ both are elements of $\mathbb{R}_=\cap\mathbb{P}=\{\mathbf{x_m},\mathbf{x_M}\}$. It is easy to check that the equivalent expressions of the sequence $\{x_k\}_{k=0}^\infty$ generated by the FPI \eqref{fpi} with any $x_0\in\mathbb{R}$ can be rewritten in the explicit form
\begin{align*}
x_{k}
=\mathbf{x_M}+\dfrac{\sqrt{5}(x_0-\mathbf{x_M})t^k}{x_0-\mathbf{x_m}-(x_0-\mathbf{x_M})t^k},
\end{align*}
 provided that $x_k$ exists for $k\geq 0$, where $t=\frac{7-3\sqrt{5}}{2}\in(-1,1)$. It is immediate that the FPI \eqref{fpi} with any $x_0$ converges to the maximal solution $\mathbf{x_M}$.
\end{example}
%
The above example observes that the existence of $\mathbf{X_m}$ (and similar to $\mathbf{Y_M}$) is not ensured by a unifying treatment
based the construction of FPI~\eqref{fpi}. This is the motivational thought that drives the study in this paper.

As previously mentioned, the aim of the paper is to investigate the existence for the minimal element $\mathbf{X_m}$ of $\mathbb{R}_= \cap \mathbb{P}$, the maximal element $\mathbf{Y_M}$ of $\mathbb{R}_= \cap \mathbb{E}$ to CDARE~\eqref{cdare}, etc. Our contributions are summarized as follows. With the aid of two auxiliary equations, we have shown that CDARE in this class~\eqref{cdare} has two extremal solutions $\mathbf{X_m}$ and $\mathbf{Y_M}$ and almost
(anti-)stabilizing solutions under certain conditions.
In addition, we prove that two extremal pairs $(\mathbf{X_M},\mathbf{X_m})$ and $(\mathbf{Y_M},\mathbf{Y_m})$
cannot be simultaneously assigned to a CDARE~\eqref{cdare}.
  Another contribution in this paper is shown that the constructive proofs for the existence and
uniqueness result of ten Hermitian solutions including $\mathbf{\mathbf{Z_M}}$, $\mathbf{Z_m}$, $\mathbf{\mathbf{W_M}}$ and $\mathbf{W_m}$ to the CDARE~\eqref{cdare} under some
mild and reasonable assumptions. For the sake of clarity, Table~\ref{Fx} lists every iterative solution with suitable initial value coincides with extremal solution or almost (anti-)stabilizing solution, respectively, as we will discuss later.
\begin{table}[tbhp] 
\begin{center}
\begin{tabular}{c|c|c|c}
\hline
Extremal solutions & General assumption & FPI & Convergence  \\
 &$A\in\mathbb{C}^{n\times n},\,G,H\in\mathbb{H}_n$  &  &  region \\
\hline\hline\\[-2ex]
 $\mathbf{X_M}=$ &$\mathbb{R}_\leq\cap\mathbb{P}\neq\emptyset$ and &$X_{k+1}=\mathcal{R}(X_k)$& $\mathbb{U}_\geq$ \\
 $\max(\mathbb{R}_=\cap\mathbb{P})$ &$\mathbb{T}\neq\emptyset$ & \\
 \hline\\[-2ex]
 $\mathbf{X_m}=$ &$\mathbb{R}_\leq\cap\mathbb{P}\neq\emptyset$, $\mathbb{O}\neq\emptyset$ &$X_{k+1}=-{\mathcal{R}^{-1}}(-X_k)$& $-{\mathbb{V}}_\geq$ \\
 $\min(\mathbb{R}_=\cap\mathbb{P})$  & and $\overline{H_A}\in\mathrm{dom}(\mathcal{R})$ &\\
\hline\\[-2ex]
 $\mathbf{Y_M}=$ &$\mathbb{R}_\geq\cap\mathbb{E}\neq\emptyset$, $\mathbb{O}\neq\emptyset$ &$X_{k+1}=-{\mathcal{R}^{-1}}(-X_k)$& $-{{\mathbb{V}}}_\leq$ \\
$\max(\mathbb{R}_=\cap\mathbb{E})$ & and $\overline{H_A}\in\mathrm{dom}(\mathcal{R})$ &\\
\hline\\[-2ex]
 $\mathbf{Y_m}=$ &$\mathbb{R}_\geq\cap\mathbb{E}\neq\emptyset$ and &$X_{k+1}=\mathcal{R}(X_k)$& $\mathbb{U}_\leq$ \\
 $\min(\mathbb{R}_=\cap\mathbb{E})$ &$\mathbb{T}\neq\emptyset$ & \\
\hline\hline\\[-2ex]
Extremal solutions & General assumption & FPI & Convergence  \\
 &$A\in\mathbb{C}^{n\times n},\,G,H\in\mathbb{N}_n$  &  &  region \\
\hline\hline\\[-2ex]
$\mathbf{Z_M}=$ &$\mathbb{T}\neq\emptyset$  & $X_{k+1}=\mathcal{R}(X_k)$& $\mathbb{U}_\geq$ \\
 $\max(\mathbb{R}_=\cap\mathbb{N}_n)$ &  & \\
 \hline\\[-2ex]
 $\mathbf{Z_m}=$ & $\mathbb{R}_\geq\cap\mathbb{N}_n\neq\emptyset$ &$X_{k+1}=\mathcal{R}(X_k)$& $[0,H]$ \\
 $\min(\mathbb{R}_=\cap\mathbb{N}_n)$ & & \\
\hline\\[-2ex]
$\mathbf{W_M}=$ &$A$ is nonsingular,
 &$X_{k+1}=-{\mathcal{R}^{-1}}(-X_k)$& $[0,-\widetilde{H}]$ \\
 $\max(\mathbb{R}_=\cap-\mathbb{N}_n)$ & and $\mathbb{R}_\geq\cap-\mathbb{N}_n\neq\emptyset$ & \\
[0.5ex]
\hline\hline\\[-2ex]
$\mathbf{W_m}=$ &$\mathbb{O}\neq\emptyset$  & $X_{k+1}=-{\mathcal{R}^{-1}}(-X_k)$& $-\overline{\mathbb{V}}_\geq$ \\
 $\min(\mathbb{R}_=\cap-\mathbb{N}_n)$ &  & \\
 \hline\hline\\[-2ex]
Upper bounds & General assumption & FPI & Convergence  \\
 &$A\in\mathbb{C}^{n\times n},\,G,H\in\mathbb{H}_n$  &  &  region \\
\hline\hline\\[-2ex]
$\mathbf{X_{a.s.}}=$ &$\mathbb{T}\neq\emptyset$, $R>0$  & $X_{k+1}=\mathcal{R}(X_k)$& $\mathbb{U}_\geq$ \\
 $\mathbf{U}(\mathbb{R}_=\cap\mathbb{P})$ &and $H\geq 0$  & \\
 \hline\\[-2ex]
 $\mathbf{X_{a.a.s.}}=$ &$\mathcal{R}$ is order preserving, &$X_{k+1}=\mathcal{R}(X_k)$& $[0,H]$ \\
 $\min(\mathbb{R}_=\cap\mathbb{N}_n)$ &$H\geq 0$ and $\mathbb{R}_\geq\cap\mathbb{N}_n\neq\emptyset$ & \\
 [0.5ex]
 \hline
 \end{tabular}
\end{center}
\caption{The constructive iterations for the existence of ten specific Hermitian solutions}\label{Fx}
\end{table}

The rest of this paper is organized as follows. In Section~\ref{sec2}, some notations and preliminaries are first provided. Next, we review briefly recent work
for which establish the existence of two extremal solutions $\mathbf{X_M}$ and $\mathbf{Y_m}$
based on the framework of FPI~\eqref{fpi}. Especially, we introduce two
auxiliary matrix equations in order to obtain another extremal solutions. Furthermore,
the convergence region techniques are used to derive some iterative solutions. Section 3 devotes to the development
of the dual-CDARE which guarantee the existence of two extremal solutions $\mathbf{X_m}$ and $\mathbf{Y_M}$, and the existence of another extremal solutions to CDAREs~\eqref{cdare} will be addressed. Two subsets $\mathbb{R}_\leq \cap \mathbb{P}$ and $\mathbb{R}_\geq \cap \mathbb{E}$ have empty intersection will be discussed {in Section~\ref{sec4} under reasonable hypotheses. The (almost) stabilizing solution and almost (anti-)stabilizing solution are studied in Section~\ref{sec4.5}. Finally, we conclude the paper in Section~\ref{conclusion}.

\section{Preliminaries}\label{sec2}
First of all, we will introduce some notations and definitions which we need in the rest of the paper.
\subsection{Some useful definitions, theory and notations}
For any $M,N\in \mathbb{H}_n$, the positive definite and positive semidefinite matrices are denoted by $M > 0$ and $M\geq 0$, respectively.
Moreover, we usually denote by $M \geq N$ (or $M \leq N$) if $M - N \geq 0$ (or $N - M \geq 0$) in the context. We use $[M,N]$ to denote the set,
$\{X\in\mathbb{H}_n|M\leq X \leq N\}$ when $N\geq M$.
For the sake of simplicity, the spectrum and spectral radius of $A\in \mathbb{C}^{n\times n}$ are
denoted by $\sigma (A)$ and $\rho (A)$, respectively, and $\mu(A) := \min \{|\lambda|\ |\ \lambda\in \sigma (A)\}$.
We denote, the complement of the open unit disk by $\mathbb{D}_{\geq}$, the closed unit disk by $\mathbb{D}_{\leq}$ and
the unit disk by $\mathbb{D}_{=}$, respectively.
Given a subset $S$ of $\mathbb{H}_n$, the following specific definitions will be employed through the paper.
\begin{Definition}\par\noindent
The following definitions are given for $S\subseteq\mathbb{H}_n$.
\begin{enumerate}
  \item $-S:=\{-X;X\in S\}$ and $\overline{S}:=\{\overline{X};X\in S\}$.
  \item $S$ is called a $\mathcal{R}$-invariant set if $\mathcal{R}(S)\subseteq S\subseteq \mathrm{dom}(\mathcal{R})$.
  \item
   $S$ is called a convergence region to a Hermitian solution $X_\ast$ under the FPI~\eqref{fpi} if the sequence $\{X_k\}_{k=0}^\infty$ with any $X_0\in S$ generated by FPI~\eqref{fpi} converge to $X_\ast$.
  \item A matrix operator $f:S\rightarrow\mathbb{H}_{n}$ is order preserving (resp. to  reversing) on $S$ if $f(A)\geq f(B)$ (resp. to $f(A)\leq f(B)$) when $A\geq B$ and $A,B\in S$.
  \item $S_U$ (resp. $S_L$) is an upper (resp. lower) bound of $S$ if $S_U\geq X$ (resp. to $S_L\leq X$) for all $X\in S$. The collection of all upper (resp. to lower) bounds of $S$ is denoted by $\mathbf{U}(S)$ (resp. to $\mathbf{L}(S)$). Moreover, $S_U=\max(S)$ (resp. $S_L=\min(S)$) is the maximal (resp. minimal) element of $S$ if $S_U\in S$ (resp. $S_L\in S$).
\end{enumerate}
\end{Definition}
We introduce the following famous results which we need in the rest of this paper. The results in the following lemma either follow immediately from the definition or are easy to verify.
\begin{Lemma}\cite{Bernstein2009}\label{Schur}
Let $A$ be an arbitrary matrix of size $n$. $X$ and $Y$ are two $n\times n$ positive definite matrices. Then,
\begin{itemize}
\item[1.]Sherman Morrison Woodbury formula (SMWF):

 Assume that $Y^{-1}\pm AX^{-1}A^H$ is nonsingular. Then, $X\pm A^H Y A$ is invertible and
\[
(X\pm A^H Y A)^{-1}=X^{-1}\mp X^{-1}A^H(Y^{-1}\pm AX^{-1}A^H)^{-1}AX^{-1}.
\]
\item[2.]A Schur complement condition for positive definiteness and positive semidefiniteness:

 A square complex matrix $\Psi$ is partitioned as
$
\Psi:=\bb X & A \\ A^H & Y\eb.
$
Then, $\Psi>0$ (resp. $\Psi\geq 0$) if and only if the Schur complement $\Psi/X:=Y-A^H X^{-1} A>0$ (resp. $Y-A^H X^{-1} A\geq0$) if and only if the Schur complement $\Psi/Y:=X-A Y^{-1}A^H>0$ (resp. $X-A Y^{-1}A^H\geq0$).
\end{itemize}
\end{Lemma}
To advance the readability of the paper more clearly, all extremal solution and almost (anti-)stabilizing solutions are defined in Table~\ref{Fx5} and Table~\ref{Fx6}, respectively. Table~\ref{Fx1} and Table~\ref{Fx2} list all useful notations arising from CDARE~\eqref{cdare} and two auxiliary matrix equations, which both are classified into three categories, respectively.
\begin{table}[tbhp]
\begin{center}
\begin{tabular}{c|c|c|c|c|c|c|c|c|c}
\hline\\[-2ex]
Subset & $\mathbb{R}_\leq\cap\mathbb{P}$ &$\mathbb{R}_\geq\cap\mathbb{E}$   & $\widetilde{\mathbb{R}}_\leq\cap\widetilde{\mathbb{P}}$ &$\widetilde{\mathbb{R}}_\geq\cap\widetilde{\mathbb{E}}$& $\widehat{\mathbb{R}}_\leq\cap\widehat{\mathbb{P}}$ &$\widehat{\mathbb{R}}_\geq\cap\widehat{\mathbb{E}}$
&${\mathbb{R}}_=\cap{\mathbb{N}_n}$& ${\mathbb{R}}_=\cap{-\mathbb{N}_n}$
&$\widetilde{\mathbb{R}}_=\cap{\mathbb{N}_n}$
 \\[0.5ex]
\hline\\[-2ex]
maximizer&$\mathbf{X_M}$&$\mathbf{Y_M}$&$\mathbf{\widetilde{X}_M}$
&$\mathbf{\widetilde{Y}_M}$&$\mathbf{\widehat{X}_M}$&$\mathbf{\widehat{Y}_M}$ &$\mathbf{{Z}_M}$&$\mathbf{{W}_M}$&$\mathbf{\widetilde{Z}_M}$
\\
\hline\\[-2ex]
minimizer&$\mathbf{X_m}$&$\mathbf{Y_m}$&$\mathbf{\widetilde{X}_m}$
&$\mathbf{\widetilde{Y}_m}$&$\mathbf{\widehat{X}_m}$&$\mathbf{\widehat{Y}_m}$  &$\mathbf{{Z}_m}$&$\mathbf{{W}_m}$&$\mathbf{\widetilde{Z}_m}$
\\
 \hline
 \end{tabular}
\end{center}
\caption{The notations of all extremal solutions in the paper}\label{Fx5}
\end{table}
\begin{table}[tbhp] 
\begin{center}
\begin{tabular}{c|c|c|c}
\hline\\[-2ex]
Equations &CDARE&dual-CDARE& transformed-DARE \\[0.5ex]\hline\\[-2ex]
 almost & $\mathbf{X_{a.s.}}\in\mathbb{R}_=$ with &$\mathbf{\widetilde{X}_{a.s.}}\in\widetilde{\mathbb{R}}_=$ with & $\mathbf{\widehat{X}_{a.s.}}\in\widehat{\mathbb{R}}_=$ with
 \\[0.5ex]\\[-2ex]
 stabilizing solution& $\rho(\widehat{T}_{\mathbf{X_{a.s.}}})\leq 1$ &$\rho(\widehat{U}_{\mathbf{\widetilde{X}_{a.s.}}})\leq 1$   & $\rho(\widehat{T}^{(D)}_{\mathbf{\widehat{X}_{a.s.}}})\leq 1$ \\[1ex]\hline\hline\\[-2ex]
 almost anti- & $\mathbf{X_{a.a.s.}}\in\mathbb{R}_=$ with &$\mathbf{\widetilde{X}_{a.a.s.}}\in\widetilde{\mathbb{R}}_=$ with & $\mathbf{\widehat{X}_{a.a.s.}}\in\widehat{\mathbb{R}}_=$ with
 \\[0.5ex]\\[-2ex]
 stabilizing solution& $\mu(\widehat{T}_{\mathbf{X_{a.a.s.}}})\geq 1$ &$\mu(\widehat{U}_{\mathbf{\widetilde{X}_{a.a.s.}}})\geq 1$   & $\mu(\widehat{T}^{(D)}_{\mathbf{\widehat{X}_{a.a.s.}}})\geq 1$

\\[1ex]\hline
 \end{tabular}
\end{center}
\caption{The notations of all (anti-)stabilizing solutions}\label{Fx6}
\end{table}

\begin{table}[tbhp] 
\begin{center}
\begin{tabular}{c|c|c}
\hline
Category & Notation &Definition  \\
\hline\\[-2ex]
&$\mathrm{dom}(\mathcal{R})$&$\{X\in \mathbb{H}_n\, |\, \det(I+GX)\neq 0\}$  \\
 &$\mathbb{R}_\leq$&$\{X\in \mathrm{dom}(\mathcal{R})\, |\, X  \leq  \mathcal{R}(X) \}$  \\
&$\mathbb{R}_=$&$\{X\in \mathrm{dom}(\mathcal{R})\, |\, X =\mathcal{R}(X) \}$ \\
 &$\mathbb{R}_\geq$&$\{X\in \mathrm{dom}(\mathcal{R})\, |\, X  \geq  \mathcal{R}(X) \}$ \\
 &$\mathbb{U}_\geq$&$\bigcup\limits_{W \in\mathbb{T}} \{ X\in \mathbb{H}_n\, |\, \mathcal{C}_{T_W} (X)\geq H_{F_W}\}$\\
Subsets of $\mathbb{H}_n$ &$\mathbb{U}_\leq$&$\bigcup\limits_{W \in\mathbb{T}} \{ X\in \mathbb{H}_n\, |\, \mathcal{C}_{T_W} (X)\leq H_{F_W}\}$\\
 &$\mathbb{V}_\geq$&$\bigcup\limits_{W \in\mathbb{O}} \{ X\in \mathbb{H}_n\, |\, \mathcal{C}_{T_W} (X)\geq H_{F_W}\}$\\
&$\mathbb{V}_\leq$&$\bigcup\limits_{W \in\mathbb{O}} \{ X\in \mathbb{H}_n\, |\, \mathcal{C}_{T_W} (X)\leq H_{F_W}\}$\\
&$\mathbb{T}$&$\{X\in\mathrm{dom}(\mathcal{R}) \,|\, \rho(\widehat{T}_X)<1\}$ \\
&$\mathbb{O}$&$\{X\in\mathrm{dom}(\mathcal{R}) \,|\, \mu(\widehat{T}_X)>1\}$ \\
 &$\mathbb{P}$&$\{X\in \mathrm{dom}(\mathcal{R}) \,|\, R_X>0\}$  \\
&$\mathbb{E}$&$\{X\in \mathrm{dom}(\mathcal{R}) \,|\, R_X<0\}$  \\
\hline\hline\\[-2ex]
&$R_X$&$R+B^H\overline{X}B$  \\
&$F_X$&$R_X^{-1}B^H \overline{X}A$ \\
Specific matrices &$T_X$&$A-BF_X$ or $(I + G \overline{X})^{-1}A$  \\
 &$\widehat{T}_X$&$\overline{T_X}T_X$ \\
 &$H_F$&$H+F^H R F$ \\
\hline\hline\\[-2ex]
Matrix operators &$\mathcal{C}_A(X)$&$X-A^H \overline{X} A$ \\
 &$\mathcal{K}_F(X)$&$(F-F_{X})^HR_X(F-F_{X})$ \\
 \hline
 \end{tabular}
\end{center}
\caption{Some particular notations arising
from CDARE~}\label{Fx1}
\end{table}

\begin{table}[tbhp] 
\begin{center}
\begin{tabular}{c|c|c}
\hline
Category & Notation &Definition  \\
\hline\\[-2ex]
&$\mathrm{dom}(\widehat{\mathcal{R}})$&$\{X\in \mathbb{H}_n\, |\, \det(I+\widehat{G}X)\neq 0\}$  \\
 &$\widehat{\mathbb{R}}_\leq$&$\{X\in \mathrm{dom}(\mathcal{\widehat{R}})\, |\, X  \leq  \widehat{\mathcal{R}}(X) \}$  \\
&$\widehat{\mathbb{R}}_=$&$\{X\in \mathrm{dom}(\mathcal{R})\, |\, X =\widehat{\mathcal{R}}(X) \}$ \\
 &$\widehat{\mathbb{R}}_\geq$&$\{X\in \mathrm{dom}(\widehat{\mathcal{R}})\, |\, X  \geq  \widehat{\mathcal{R}}(X) \}$ \\
 &$\widehat{\mathbb{U}}_\geq$&$\bigcup\limits_{W \in\widehat{\mathbb{T}}} \{ X\in \mathbb{H}_n\, |\, \mathcal{S}_{\widehat{T}^{(D)}_W} (X)\geq \widehat{H}_{\widehat{F}_W}\}$\\
Subsets of $\mathbb{H}_n$ &$\widehat{\mathbb{U}}_\leq$&$\bigcup\limits_{W \in\widehat{\mathbb{T}}} \{ X\in \mathbb{H}_n\, |\, \mathcal{S}_{\widehat{T}^{(D)}_W} (X)\leq \widehat{H}_{\widehat{F}_W}\}$\\
 &$\widehat{\mathbb{V}}_\geq$&$\bigcup\limits_{W \in\widehat{\mathbb{O}}} \{ X\in \mathbb{H}_n\, |\, \mathcal{S}_{\widehat{T}^{(D)}_W} (X)\geq \widehat{H}_{\widehat{F}_W}\}$\\
&$\widehat{\mathbb{V}}_\leq$&$\bigcup\limits_{W \in\widehat{\mathbb{O}}} \{ X\in \mathbb{H}_n\, |\, \mathcal{S}_{\widehat{T}^{(D)}_W} (X)\leq \widehat{H}_{\widehat{\widehat{F}}_W}\}$\\
&$\widehat{\mathbb{T}}$&$\{X\in\mathrm{dom}(\widehat{\mathcal{R}}) \,|\, \rho(\widehat{T}^{(D)}_X)<1\}$ \\
&$\widehat{\mathbb{O}}$&$\{X\in\mathrm{dom}(\widehat{\mathcal{R}}) \,|\, \mu(\widehat{T}^{(D)}_X)>1\}$ \\
 &$\widehat{\mathbb{P}}$&$\{X\in \mathrm{dom}(\widehat{\mathcal{R}}) \,|\, \widehat{R}_X>0\}$  \\
&$\widehat{\mathbb{E}}$&$\{X\in \mathrm{dom}(\widehat{\mathcal{R}}) \,|\, \widehat{R}_X<0\}$  \\
\hline\hline\\[-2ex]
&$\widehat{R}_X$&$\widehat{R}+\widehat{B}^H{X}\widehat{B}$  \\
&$\widehat{F}_X$&$\widehat{R}_X^{-1}\widehat{B}^H {X}\widehat{A}$ \\
Specific matrices &$\widehat{T}^{(D)}_X$&$\widehat{A}-\widehat{B}\widehat{F}_X$ or $(I + \widehat{G} {X})^{-1}\widehat{A}$  \\[0.2ex]
  &$\widehat{H}_F$&$\widehat{H}+F^H \widehat{R} F$ \\
\hline\hline\\[-2ex]
Matrix operators &$\mathcal{S}_{\widehat{A}}(X)$&$X-\widehat{A}^H {X} \widehat{A}$ \\
 &$\widehat{\mathcal{K}}_F(X)$&$(F-\widehat{F}_{X})^H\widehat{R}_X(F-\widehat{F}_{X})$ \\
 \hline
 \hline
 \hline\\[-2ex]
&$\mathrm{dom}(\widetilde{\mathcal{R}})$&$\{X\in \mathbb{H}_n\, |\, \det(I+\widetilde{G}X)\neq 0\}$  \\
 &$\widetilde{\mathbb{R}}_\leq$&$\{X\in \mathrm{dom}(\widetilde{\mathcal{R}})\, |\, X  \leq  {\widetilde{\mathcal{R}}}(X) \}$\\
&$\widetilde{\mathbb{R}}_=$&$\{X\in \mathrm{dom}(\widetilde{\mathcal{R}})\, |\, X ={\widetilde{\mathcal{R}}}(X) \}$ \\
 &$\widetilde{\mathbb{R}}_\geq$&$\{X\in \mathrm{dom}(\widetilde{\mathcal{R}})\, |\, X  \geq  {\widetilde{\mathcal{R}}}(X) \}$ \\
 &$\widetilde{\mathbb{U}}_\geq$&$\bigcup\limits_{W \in\widetilde{\mathbb{T}}} \{ X\in \mathbb{H}_n\, |\, \mathcal{C}_{\widetilde{T}_W} (X)\geq \widetilde{H}_{\widetilde{F}_W}\}$\\
Subsets of $\mathbb{H}_n$ &$\widetilde{\mathbb{U}}_\leq$&$\bigcup\limits_{W \in\widetilde{\mathbb{T}}} \{ X\in \mathbb{H}_n\, |\, \mathcal{C}_{\widetilde{T}_W} (X)\leq \widetilde{H}_{\widetilde{F}_W}\}$\\
 &$\widetilde{\mathbb{V}}_\geq$&$\bigcup\limits_{W \in\widetilde{\mathbb{O}}} \{ X\in \mathbb{H}_n\, |\, \mathcal{C}_{\widetilde{T}_W} (X)\geq \widetilde{H}_{\widetilde{F}_W}\}$\\
&$\widetilde{\mathbb{V}}_\leq$&$\bigcup\limits_{W \in\widetilde{\mathbb{O}}} \{ X\in \mathbb{H}_n\, |\, \mathcal{C}_{\widetilde{T}_W} (X)\leq \widetilde{H}_{\widetilde{F}_W}\}$\\
&$\widetilde{\mathbb{T}}$&$\{X\in\mathrm{dom}({\widetilde{\mathcal{R}}}) \,|\, \rho(\widetilde{T}_X)<1\}$ \\
&$\widetilde{\mathbb{O}}$&$\{X\in\mathrm{dom}({\widetilde{\mathcal{R}}}) \,|\, \mu(\widetilde{T}_X)>1\}$ \\
 &$\widetilde{\mathbb{P}}$&$\{X\in \mathrm{dom}(\widetilde{\mathcal{R}}) \,|\, \widetilde{R}_X>0\}$  \\
&$\widetilde{\mathbb{E}}$&$\{X\in \mathrm{dom}(\widetilde{\mathcal{R}}) \,|\, \widetilde{R}_X<0\}$  \\
\hline\hline\\[-2ex]
&$\widetilde{R}_X$&$\widetilde{R}+\widetilde{B}^H{X}\widetilde{B}$  \\
&$\widetilde{F}_X$&$\widetilde{R}_X^{-1}\widetilde{B}^H {X}\widetilde{A}$ \\
Specific matrices &$U_X$&$\widetilde{A}-\widetilde{B}\widetilde{F}_X$ or $(I + \widetilde{G} {X})^{-1}\widetilde{A}$  \\
 &$\widehat{U}_X$&$U_X\overline{U}_X$ \\
 &$\widetilde{H}_F$&$\widetilde{H}+F^H \widetilde{R} F$ \\
 \hline
 \end{tabular}
\end{center}
\caption{ Some particular notations arising
from two auxiliary equations}\label{Fx2}
\end{table}

\subsection{Two subsets $\mathbb{T}$ and $\mathbb{O}$}
In this subsection, two particular subsets of $\mathbb{H}_n$ are
introduced to make the convergence region, which play an
important role in the construction of all extremal solution. Given $X\in\mathrm{dom}(\mathcal{R})$, we denote $T_X$ by $R_X^{-1}A$
and $\widehat{T}_X=\overline{T_X}T_X$. Let
\begin{align}\label{TO}
\mathbb{T}:=\{X\in\mathrm{dom}(\mathcal{R}) \,|\, \rho(\widehat{T}_X<1\}\quad\mbox{and}\quad\mathbb{O}:=\{X\in\mathrm{dom}(\mathcal{R}) \,|\, \mu(\widehat{T}_X)>1\}.
\end{align}
The first proposition outline the useful identities concerning CDARE~\eqref{cdare}.
\begin{Proposition}\cite{CHIANG202471,MR4487984} \label{lem2p1} \color{black}
Let $X\in\mathrm{dom}(\mathcal{R})$. The following two identities hold:
\begin{subequations}\label{Req}
\begin{enumerate}
  \item For any $F\in\mathbb{C}^{m\times n}$, then
  \begin{align}
  X-\mathcal{R}(X)&=\mathcal{C}_{A_{F}}(X)-H_F+\mathcal{K}_F(X), \label{Req-a}
  \end{align}
  where $A_{F}:=
  A-BF$, $H_F:=H+F^H R F$, $\mathcal{K}_F(X) := (F-F_{X})^HR_X(F-F_{X})$ and $F_X=R_X^{-1}B^H\overline{X}A$.
  {\color{black}
  \item For any $Y\in\mathrm{dom}(\mathcal{R})$, then
  \begin{align}
  X-\mathcal{R}(X)&=\mathcal{C}_{T_{Y}}(X)-H+(\mathcal{K}(Y,X)-\mathcal{K}(Y,0)) , \label{Req-b}
  \end{align}
where $\mathcal{K}(Y,X):=K_{F_Y}(X)$. 
}
%


\end{enumerate}
\end{subequations}
\end{Proposition}
The following result
will be used frequently to
guarantee the nonemptyness of $\mathbb{T}$ and $\mathbb{O}$.
\begin{Lemma}\label{lem2p}
Let $B\in \mathbb{C}^{n\times n}$ and $Q\geq 0$. Suppose that  $\mathcal{C}_B(X)\geq Q$ for some $X\in\mathbb{H}_n$, the following statements hold.
\begin{enumerate}
\item
  Assume that
\begin{align}\label{5}
\mathrm{Ker}(Q)\cap\mathrm{Ker}(B^H\overline{Q}B)\subseteq \mathrm{Ker}(\overline{B}B-A)
\end{align}
for some $A\in\mathbb{C}^{n\times n}$, then
\begin{enumerate}
  \item $\rho(\overline{B}B)\leq \min\{1,\rho(A)\}$ if $X\geq 0$.
  \item $\mu(\overline{B}B)\geq \min\{1,\mu(A)\}$ if $X\leq 0$.
\end{enumerate}

  %
\item
Let $F_\lambda:=\{\lambda\in\mathbb{C}\, |\,\mathrm{Ker}(Q)\cap\mathrm{Ker}(B^H\overline{Q}B)\cap E_\lambda(\overline{B}B)=\{0\}\}$. Then,
\begin{enumerate}
  \item $\rho(\overline{B}B)<1$ if $X\geq 0$ and $F_\lambda\subseteq\mathbb{D}_\geq$.
  \item $\mu(\overline{B}B)>1$ if $X\leq 0$ and $F_\lambda\subseteq\mathbb{D}_\leq$.
\end{enumerate}
\end{enumerate}
\end{Lemma}
\begin{proof}
Let $\lambda\in\sigma(\overline{B}B)$ and an eigenvector $x\in E_\lambda(\overline{B}B)$. Obviously,
\begin{align}\label{8}
  x^H \mathcal{S}_{\overline{B}B}(X) x=(1-|\lambda|^2)(x^H X x) \geq x^H Q x +x^H B^H \overline{Q} Bx\geq 0.
\end{align}
\begin{enumerate}
  \item
  We give the proof only to the case where $X\geq 0$; the other cases are left to the reader. It is easy to check that
    $|\lambda|< 1$ if  $x\not\in\mathrm{Ker}(Q)\cap\mathrm{Ker}(B^H\overline{Q}B)$. Otherwise, $Ax=\overline{B}B x+(\overline{B}B-A)x=\overline{B}B x=\lambda x$ and therefore $|\lambda|\leq \rho(A)$ since the assumption \eqref{5} holds.
The remaining statement is clearly true from the above discussion.
  \item
In the case of part~(b), it is easily to see that
\[
E_\lambda(\overline{B}B)\subseteq\mathrm{Ker}(Q)\cap\mathrm{Ker}(B^H\overline{Q}B)
\]
for all $\lambda\in\sigma(\overline{B}B)\cap\mathbb{D}_\leq$ by
using \eqref{8}. Therefore, $E_\lambda(\overline{B}B)=\{0\}$ and thus $\lambda\in\mathbb{D}_>$. We complete the proof of the part~(b). The same conclusion can be drawn for the part~(a).
\end{enumerate}

\end{proof}
In particular, it seems that the assumption $\mathbb{T}\neq\emptyset$ or $\mathbb{O}\neq\emptyset$  is not easy to check. The next proposition provides some criterions to use in the paper.
\begin{Proposition}\label{pro4}
\par\noindent
\begin{enumerate}
  \item
Assume that
\begin{align}\label{c3p3}
\mathcal{C}_{{T}_X}(Y)\geq {K}(Z,X),
\end{align}
for some $X\in\mathbb{P}$ and $Y,Z\in\mathbb{H}_n$. Then,
  \begin{enumerate}
  \item $X\in{\mathbb{T}}\cap{\mathbb{P}}$, when $Y\geq 0$ and $Z\in{\mathbb{T}}$.
  \item $X\in{\mathbb{O}}\cap{\mathbb{P}}$, when $Y\leq 0$ and $Z\in{\mathbb{O}}$.
\end{enumerate}
  \item
Assume that
\begin{align}\label{c3p4}
\mathcal{C}_{{T}_X}(Y)\leq {K}(Z,X),
\end{align}
for some $X\in\mathbb{E}$ and $Y,Z\in\mathbb{H}_n$. Then,
  \begin{enumerate}
  \item $X\in{\mathbb{T}}\cap{\mathbb{E}}$, when $Y\leq 0$ and $Z\in{\mathbb{T}}$.
  \item $X\in{\mathbb{O}}\cap{\mathbb{E}}$, when $Y\geq 0$ and $Z\in{\mathbb{O}}$.
\end{enumerate}
\end{enumerate}
\end{Proposition}
\begin{proof}
First, notice that from the definition of $K(X,Y)$ we find that $\mathrm{Ker}(K(X,Y))=\mathrm{Ker}(F_X-F_Y)\subseteq\mathrm{Ker}(BF_X-BF_Y)
=\mathrm{Ker}(T_X-T_Y)\subseteq\mathrm{Ker}(\overline{T_X}(T_X-T_Y)
+(\overline{T_X}-\overline{T_Y})T_Y)=\mathrm{Ker}(\widehat{T}_X-\widehat{T}_Y)$ for any $X,Y\in\mathrm{dom}(\mathcal{R})$.
The result is a immediate consequence of the part~1 of Theorem~\ref{lem2p}.
\end{proof}
An useful sufficient condition when $R>0$ and $H\geq 0$ to determine the shape of $\mathbb{T}$ or $\mathbb{O}$ can be written as follows.
\begin{Proposition}\label{pro2}
Assume that $R>0$ and $H\geq 0$. Let $X\in\mathbb{R}_\geq\cap\mathbb{N}_n$ (resp. $-\mathbb{N}_n$). Then,
$X\in\mathbb{T}$(resp.  $X\in\mathbb{O}$) if
\[
\mathrm{rank} [\widehat{T}_X - \lambda I\ H] = n
\]
for all $\lambda \in \sigma(\widehat{T}_X)\cap\mathbb{D}_\geq$ (resp.  $\mathbb{D}_\leq$).
In addition, $(\mathbb{R}_\geq\cap\mathbb{N}_n)\cup(\mathbb{R}_\geq\cap-\mathbb{N}_n)\subseteq\mathbb{T}$ if $H>0$.
\end{Proposition}
\begin{proof}
We only need to consider the first case. Note that $X\in \mathbb{R}_\geq$ is equivalent to
\begin{align*}
 \mathcal{C}_{T_X}(X)\geq (\bar{X}T_X)^H G (\bar{X}T_X)+H.
 \end{align*}
 Therefore, $X\geq 0$ satisfying $\mathcal{C}_{T_X}(X)\geq H \geq 0$. The first conclusion immediately follows by the part~2 of Theorem~\ref{lem2p}.
\end{proof}
\subsection{Two auxiliary matrix equations}
This subsection concerns two auxiliary matrix equations, as we will later on prove that the existence of some extremal solution.
\subsubsection{The dual-CDARE }
\begin{Definition}\label{def}
Assume that $A$ is nonsingular and $\overline{H}_A\in\mathrm{dom}(\mathcal{R})$, where $H_A:=A^{-H} H A^{-1}$.
Let $\mathbb{S}_n:=\mathbb{C}^{n\times n}\times\mathbb{H}_n\times\mathbb{H}_n$. The definition of the matrix operator $\mathcal{F}:\mathbb{S}_n\rightarrow\mathbb{S}_n$ is $\mathcal{F}(A ,G, H)=(\widetilde{A},\widetilde{G},\widetilde{H})$ for any $(A,G,H)\in\mathbb{S}_n$, where
\begin{subequations}\label{AGH}
\begin{align}
 \widetilde{A}&:= A^{-1}(I+GH_A)^{-1},\label{A1}\\
  \widetilde{G}&:= \widetilde{A}G {A}^{-H}(\equiv{A}^{-1}G \widetilde{A}^{H}),\label{G1}\\
 \widetilde{H}&:= \widetilde{A}^H H A^{-1}(\equiv A^{-H} H \widetilde{A}).\label{H1}
\end{align}
\end{subequations}
\end{Definition}
By using Definition~\ref{def}, it is easily checked that
$\mathcal{F}$ is an involution operator. That is, $\mathcal{F}^{(2)}= \mathcal{F}\circ\mathcal{F}$ is the identity map on $\mathbb{S}_n$. Indeed, we introduce the notation $$(\mathcal{F}^{(2)}_A,\mathcal{F}^{(2)}_G,\mathcal{F}^{(2)}_H):=\mathcal{F}^{(2)}(A ,G, H)=\mathcal{F}(\widetilde{A},\widetilde{G},\widetilde{H}).$$
Then, we have
\begin{align*}
  &\mathcal{F}^{(2)}_A =\widetilde{A}(I+\widetilde{G}\widetilde{H}_{\widetilde{A}})^{-1}=
(I+GH_A)A(I+\widetilde{A}GA^{-H}H+\widetilde{A}GA^{-H}HA^{-1}GA^{-H}H)^{-1}\\
&=(I+GH_A)A(I+GA^{-H}H+GA^{H}HA^{-1}GA^{-H}H)^{-1}\widetilde{A}^{-1}\\
&=(I+GH_A)(I+GH_A+GA^{-H}H^{-1}A^{-1}+GH_AGH_A)^{-1}(I+GH_A)A=A,\\
&\mathcal{F}^{(2)}_G =\mathcal{F}^{(2)}_A\widetilde{G}\widetilde{A}^{H}
=AA^{-1}(I+GH_A)^{-1}GA^{-H}A^H(I+H_AG)=G,\\
&\mathcal{F}^{(2)}_H=(\mathcal{F}^{(2)}_A)^H \widetilde{H} \widetilde{A}^{-1}=A^H(A^{-H}HA^{-1})(I+GH_A)A=H,
\end{align*}
by using
\begin{align*}
\widetilde{H}_{\widetilde{A}}=(A^H(I+H_AG))[(A^{-1}(I+GH_A)^{-1})^HHA^{-1}]((I+GH_A)A)
=H(I+A^{-1}GA^{-H}H)=H+HG_{A^H}H.
\end{align*}
The idea behind the coefficient matrices \eqref{AGH} provides the formulation of the dual-CDARE :
   \begin{align}\label{dual}
{Y}=\overline{\widetilde{\mathcal{R}}}(Y)
   \end{align}
   with
   $$\widetilde{\mathcal{R}}(Y):=\widetilde{H}+\widetilde{A}^H Y(I+\widetilde{G}Y)^{-1}\widetilde{A}\equiv  \widetilde{H}+\widetilde{A}^H {Y}\widetilde{A} - \widetilde{A}^H {Y}\widetilde{B}(\widetilde{R}+\widetilde{B}^H {Y}\widetilde{B})^{-1} \widetilde{B}^H {Y}\widetilde{A},$$
    where $\widetilde{B} = A^{-1}B$ and
   $\widetilde{R}= R + B^H  H_A B$, which play an important role throughout the paper. According to Definition~\ref{def}, the following important feature for the dual-CDARE~ \eqref{dual} can be readily obtained.
\begin{Proposition}\label{pro3}
Assume that $A$ is nonsingular. Then,
\begin{enumerate}
  \item $\overline{\widetilde{H}}_{\widetilde{A}}\in\mathrm{dom}(\widetilde{\mathcal{R}})$ if and only if  $\overline{H}_A\in\mathrm{dom}(\mathcal{R})$, where $\widetilde{H}_{\widetilde{A}}:=\widetilde{A}^{-H}\widetilde{H}\widetilde{A}^{-1}$.
  \item
   Assume that $\overline{H}_A\in\mathrm{dom}(\mathcal{R})$. Then,
   \begin{enumerate}
   \item
   $\mathcal{F}$ is an involution function, i.e., $\mathcal{F}^{(2)}=I_{\mathbb{S}_n}$, which is the identity mapping on $\mathbb{S}_n$. In other words,
  \begin{align*}
 {A}&= \widetilde{A}^{-1}(I+\widetilde{G}\widetilde{H}_{\widetilde{A}})^{-1}, \\
 {G}&= {A}\widetilde{G}{\widetilde{A}}^{-H}=\widetilde{A}^{-1}\widetilde{G} {A}^{H}, \\
 {H}&= {A}^H \widetilde{H} \widetilde{A}^{-1}=\widetilde{A}^{-H} \widetilde{H} {A}.
\end{align*}
  \item
$X\in\mathrm{dom}(\mathcal{R})$ if and only if $\mathcal{R}(X)\in\mathrm{dom}(\widetilde{\mathcal{R}})$. Moreover,
\begin{align*}
(\widetilde{\mathcal{R}}\circ(-\mathcal{R}))(X)=-\overline{X},\,\forall X\in\mathrm{dom}(\mathcal{R}),\\
(\mathcal{R}\circ(-\overline{\widetilde{\mathcal{R}}}))(Y)=-{Y},\,\forall Y\in\mathrm{dom}(\widetilde{\mathcal{R}}).
\end{align*}
\item
The matrix operator  $\mathcal{R}^{-1}:-\mathrm{dom}(\widetilde{\mathcal{R}})\rightarrow \mathbb{H}_n$ is defined by
    $\mathcal{R}^{-1}(Y)=-\overline{\widetilde{\mathcal{R}}}(-Y)$ for any $Y\in -\mathrm{dom}(\widetilde{\mathcal{R}})$, satisfying $\mathcal{R}^{-1}\circ\mathcal{R}=I$. Similarly, the matrix operator  $\widetilde{\mathcal{R}}^{-1}:-\mathrm{dom}(\mathcal{R})    \rightarrow \mathbb{H}_n $ is defined by
    $\widetilde{\mathcal{R}}^{-1}(X)=-\mathcal{R}(-\overline{X})$ for any $X\in \mathrm{dom}(\mathcal{R})$, satisfying $\widetilde{\mathcal{R}}^{-1}\circ\widetilde{\mathcal{R}}=I$.
\end{enumerate}
\end{enumerate}
\end{Proposition}
\begin{proof}
\par\noindent
\begin{enumerate}
\item
This result follows immediately from
\begin{align*}
I+\widetilde{G}\widetilde{H}_{\widetilde{A}}=I+A^{-1}G\widetilde{A}^H
\widetilde{A}^{-H}\widetilde{H}\widetilde{A}^{-1}=I+A^{-1}GA^{-H}H=A^{-1}(I+GH_A)A.
\end{align*}

%
  \item
 \begin{enumerate}
   \item
The result has been proved previously.
 \item
Note that 
\begin{align*}
  &I+G\overline{X}-A(I-\widetilde{G} H)^{-1}\widetilde{G}A^H\overline{X}=I+(G-\widetilde{A}^{-1}\widetilde{G}A^H)\overline{X}=I,\\
&(I-\widetilde{G}H)^{-1}=A^{-1}\widetilde{A}^{-1}.
\end{align*}
Therefore, $I-\widetilde{G}\mathcal{R}(X)$ is invertible by SMWF and its inversion can be presented in the following form:
\begin{align*}
&(I-\widetilde{G}\mathcal{R}(X))^{-1}=\left(I-\widetilde{G}H-\widetilde{G}A^{H}\overline{X}(I+G\overline{X})^{-1}A\right)^{-1}
=A^{-1}\widetilde{A}^{-1}+A^{-1}\widetilde{A}^{-1}\widetilde{G}A^{H}\overline{X}(I)^{-1}AA^{-1}\widetilde{A}^{-1}\\
&=A^{-1}(I+G\overline{X})\widetilde{A}^{-1}.
\end{align*}
Thus,
\begin{align*}
&(\widetilde{\mathcal{R}}\circ(-\mathcal{R}))(X)
=\widetilde{H}
-\widetilde{A}^H\mathcal{R}(X)(I-\widetilde{G}\mathcal{R}(X))^{-1}\widetilde{A}
=\widetilde{A}^H(H-\mathcal{R}(X)A^{-1}(I+G\overline{X})A)A^{-1}\\
&=\widetilde{A}^H(H-HA^{-1}(I+G\overline{X})A-A^H\overline{X}A)A^{-1}
=\widetilde{A}^H(-A^H(I+H_AG)\overline{X}A)A^{-1}=-\overline{X}.
\end{align*}
Note that
\begin{align*}
  &I-G\widetilde{H}=I-G(I+H_AG)^{-1}A^{-H}HA^{-1}=(I+GH_A)^{-1},\\
  &I+\widetilde{G}{Y}-\widetilde{A}(I-{G} \widetilde{H})^{-1}{G}\widetilde{A}^H{Y}=I+(\widetilde{G}-{A}^{-1}{G}\widetilde{A}^H){Y}=I.
\end{align*}
Therefore, $I-{G}\widetilde{\mathcal{R}}(Y)$ is invertible by SMWF and its inversion can be presented in the following form:
\begin{align}
&(I-{G}\widetilde{\mathcal{R}}(Y))^{-1}=\left(I-G\widetilde{H}-{G}\widetilde {A}^{H}{Y}(I+\widetilde{G}{Y})^{-1}\widetilde{A}\right)^{-1}
=A^{-1}\widetilde{A}^{-1}+A^{-1}\widetilde{A}^{-1}\widetilde{G}A^{H}\overline{Y}(I)^{-1}AA^{-1}\widetilde{A}^{-1}\nonumber\\
&=A^{-1}(I+G\overline{Y})\widetilde{A}^{-1}.\label{231}
\end{align}
Taking into account \eqref{231} we have
\begin{align*}
&(\mathcal{R}\circ(-\overline{\widetilde{\mathcal{R}}}))(Y)={H}-{A}^H\widetilde{\mathcal{R}}(Y)(I-{G}\widetilde{\mathcal{R}}(Y))^{-1}{A}
={H}-{A}^H\widetilde{\mathcal{R}}(Y)\widetilde{A}^{-1}(I+\widetilde{G}(Y))\\
&=H-A^H(\widetilde{H}+\widetilde{A}^HY(I+\widetilde{G}Y)^{-1}\widetilde{A})\widetilde{A}^{-1}(I+\widetilde{G}(Y))
=H-A^H\widetilde{H}\widetilde{A}^{-1}(I+\widetilde{G}Y)-A^H\widetilde{A}^HY\\
&=-(H\widetilde{G}+A^H\widetilde{A}^H)Y=-{Y}.
\end{align*}
\end{enumerate}
\end{enumerate}
This completes the proof.
\end{proof}
\subsubsection{The transformed-DARE }
In analogy with the technique in \cite{CHIANG202471,l.c18}, if $\det(R_H)\neq 0$, then the CDARE~\eqref{cdare-b}
can be transformed into a discrete-time algebraic Riccati equation (DARE) of the form
\begin{subequations}\label{coeff-dare2}
\begin{align}
X = \widehat{\mathcal{R}}(X) &:= \mathcal{R}(\mathcal{R}(X)) = \widehat{A}^H {X}\widehat{A} - \widehat{A}^H {X}\widehat{B}(\widehat{R}+
\widehat{B}^H {X}\widehat{B})^{-1} \widehat{B}^H {X} \widehat{A} + \widehat{H}, \label{dare2-a} 
\end{align}
where its coefficient matrices corresponding the standard form of CDARE~\eqref{cdare-a} are given by
  \begin{align}
\widehat{A}&=\overline{A}T_H=\overline{A}(A-BF_H),
\widehat{B} =\bb \overline{B} & \overline{A}B \eb, \label{hatA} \widehat{R} = \overline{R}\oplus R_H, \\
\widehat{G}&= \widehat{B} \widehat{R}^{-1} \widehat{B}^H = \overline{G}+\overline{A} (I+G\overline{H})^{-1}{G}\overline{A}^H,
\widehat{H} = H+{A}^H \overline{H}(I+G\overline{H})^{-1}{A}. \label{hatG} 
\end{align}
\end{subequations}
The following lemma characterizes a useful identity depending on the matrix operator $\widehat{\mathcal{R}}(\cdot)$ and its
associated Stein operator.
\begin{Lemma}\label{lemma23}\cite{MR4794239}
Assume that $H\in\mathrm{dom}({\mathcal{R}})$.
 For any $X,Y\in\mathrm{dom}(\widehat{\mathcal{R}})$, we have
  \begin{align}
  X-\widehat{\mathcal{R}}(X)&=\mathcal{S}_{\widehat{T}_{Y}}(X)-\widehat{H}+(\widehat{\mathcal{K}}(Y,X)-\widehat{\mathcal{K}}(Y,0)) , \label{Req-c}
  \end{align}
where $\widehat{\mathcal{K}}(Y,X) :=(\widehat{F}-\widehat{F}_{X})^H\widehat{R}_X(\widehat{F}-\widehat{F}_{X})$.
\end{Lemma}
\subsection{The FPI~\eqref{fpi} with its convergence property}
Very recently, it was shown in \cite[Lemma 2.3]{MR4487984} that every element of $\mathbb{U}_{\geq}$ is an upper bound of $\mathbb{R}_\leq\cap\mathbb{P}$. A result analogous to the fact that every element of $\mathbb{U}_{\leq}$ is a lower bound of $\mathbb{R}_\geq\cap\mathbb{E}$. Moreover, the existence of the maximal element to $\mathbb{R}_\leq\cap\mathbb{P}$ has been established iteratively, utilizing the framework of the FPI~\eqref{fpi} with an appropriate initial matrix $X_0$. In fact, this maximizer solve CDARE~\eqref{cdare} and thus is exactly equal to $\mathbf{X_M}$. In order to deduce our main theorem presented in the next section, the previous result is summarized as follows, and the proof procedure for the existence of minimizer of $\mathbb{R}_\geq\cap\mathbb{E}$ is very similar
 can be proved in the same way as for the maximizer of $\mathbb{R}_\leq\cap\mathbb{P}$, and thus the proof is omitted.
\begin{Theorem} \label{thm3p1}
The following statements hold:
\begin{enumerate}
  \item \cite[Theorem 3.1]{MR4487984}
When $\mathbb{R}_{\leq}\cap\mathbb{P}\neq\emptyset$, then $\mathbb{R}_{\leq}\cap\mathbb{P}$ admits a maximal element which is exactly equal to $\mathbf{X_M}$. Moreover,
  \begin{enumerate}
  \item[(i)]
The sequence $\{X_k\}_{k=0}^\infty$ generated by the FPI \eqref{fpi} with $X_0\in\mathbb{U}_\geq$ is well-defined. Moreover,
   $X_k\in\mathbb{U}_\geq \subseteq \mathbb{R}_\geq \cap \mathbb{P} \cap \mathbb{T}$ for all $k\geq 0$.
  \item[(ii)] $X_k\geq X_{k+1}\geq X_\mathbb{P}$ for all $k\geq 0$ and $X_\mathbb{P}\in \mathbb{R}_{\leq} \cap \mathbb{P}$.
  \item[(iii)] The sequence $\{X_k\}_{k=0}^\infty$  converges to $\mathbf{X_M}$, which is the maximal element of the set $\mathbb{R}_\leq\cap \mathbb{P}$
and satisfies $\rho (\widehat{T}_{\mathbf{X_M}}) \leq 1$, with the rate of convergence:
  $$  \limsup_{k\rightarrow \infty} \sqrt[k]{\|X_k - \mathbf{X_{M}}\|}\leq \rho (\widehat{T}_{\mathbf{X_M}})^2.$$
Moveover, the convergence is at least linearly, whenever $\mathbf{X_M}\in\mathbb{T}$.
%
\end{enumerate}
  \item
  When $\mathbb{R}_{\geq}\cap\mathbb{E}\neq\emptyset$, then $\mathbb{R}_{\geq}\cap\mathbb{E}$ admits a minimal element which is exactly equal to $\mathbf{Y_m}$. Moreover,
  \begin{enumerate}
  \item[(i)]
The sequence $\{X_k\}_{k=0}^\infty$ generated by the FPI~\eqref{fpi} with $X_0\in\mathbb{U}_\leq$ is well-defined. Moreover,
   $X_k\in\mathbb{U}_\leq \subseteq \mathbb{R}_\leq \cap \mathbb{E} \cap \mathbb{T}$ for all $k\geq 0$.
  \item[(ii)] $X_k\leq X_{k+1}\leq X_\mathbb{E}$ for all $k\geq 0$ and $X_\mathbb{E}\in \mathbb{R}_{\geq} \cap \mathbb{E}$.
  \item[(iii)] The sequence $\{X_k\}_{k=0}^\infty$  converges to $\mathbf{Y_m}$, which is the minimal element of the set $\mathbb{R}_\geq\cap \mathbb{E}$
and satisfies $\rho (\widehat{T}_{\mathbf{Y_m}}) \leq 1$, with the rate of convergence:
    $$\limsup_{k\rightarrow \infty} \sqrt[k]{\|X_k - \mathbf{Y_{m}}\|}\leq \rho (\widehat{T}_{\mathbf{Y_m}})^2. $$
    Moveover, the convergence is at least linearly, whenever $\mathbf{Y_m}\in\mathbb{T}$.
\end{enumerate}
\end{enumerate}
\end{Theorem}

Next, a preliminary is introduced before giving the main result in this subsection.
Inspired by the idea of the convergence region under FPI~\eqref{fpi}, we present an alternative approach to Lemma~\ref{thm3p1}. The following result gives a special expansion of Lemma~\ref{thm3p1}, which plays a crucial role in the convergence of FPI~\eqref{fpi}.
\begin{Lemma}\label{11}
Assume that $S\subseteq \mathbb{R}_\geq\cap\mathbf{U}(\mathbb{R}_=\cap K)$ (resp. $S\subseteq \mathbb{R}_\leq\cap\mathbf{L}(\mathbb{R}_=\cap K)$) for two subsets $S$ and $K$ of $\mathbb{H}_n$. Then, the FPI~\eqref{fpi} generates a nondecreasing (resp. nonincreasing) sequence of matrices
$\{X_k\}$ with any $X_0\in S$ , which converges to an element of $\mathbb{R}_=\cap \mathbf{U}(\mathbb{R}_=\cap K)$ (resp. $\mathbb{R}_=\cap \mathbf{L}(\mathbb{R}_=\cap K)$), if either $\mathcal{R}$ is order preserving on $\mathrm{dom}(\mathcal{R})$ or $S$ is $\mathcal{R}$-invariant.
\end{Lemma}
\begin{proof}
Its proof is based on the completeness of the set of complex matrices~\cite[Proposition 8.6,3.]{Bernstein2009}. Namely, every sequence of complex matrices
increases (resp. decreases) and is bounded above (resp. bounded below) by a supremum (resp. infimum), it will converge to the supremum (resp. infimum). The result can be obtained by assumptions immediately.
\end{proof}
With the preceding preliminary, one can obtain the following proposition on the convergence region to $\mathbf{X_M}$ and $\mathbf{Y_m}$ under FPI~\eqref{fpi}, and the lower bound of $\mathbb{R}_\leq\cap\mathbb{P}$ and upper bound of $\mathbb{R}_\geq\cap\mathbb{E}$, respectively.
\begin{Proposition}\label{P1}
\begin{subequations}
\par\noindent
\begin{enumerate}
\item Assume that $\mathbb{T}\neq\emptyset$. 
Then,
 \begin{align}
&(\mathbb{R}_\geq\cap\mathbb{E})\cup\mathcal{R}(\mathbb{U}_\geq)\subseteq\mathbb{U}_\geq\subseteq\mathbf{U(\mathbb{R}_\leq\cap\mathbb{P})}\cap(\mathbb{R}_\geq\cap\mathbb{P})\cap\mathbb{T}
,\,\mbox{if}\quad\mathbb{R}_\leq\cap\mathbb{P}\neq\emptyset.\label{cr1}\\
&(\mathbb{R}_\leq\cap\mathbb{P})\cup\mathcal{R}(\mathbb{U}_\leq)\subseteq\mathbb{U}_\leq\subseteq \mathbf{L}(\mathbb{R}_\geq\cap\mathbb{E})\cap(\mathbb{R}_\leq\cap\mathbb{E})\cap\mathbb{T}
,\,\mbox{if}\quad\mathbb{R}_\geq\cap\mathbb{E}\neq\emptyset.\label{cr2}
\end{align}
  \item Assume that $\mathbb{O}\neq\emptyset$. Then,
\begin{align}
&\mathbb{R}_\geq\cap\mathbb{E}\subseteq{\mathbb{V}}_\geq\subseteq\mathbf{L(\mathbb{R}_\leq\cap\mathbb{P})},\,\mbox{if}\quad\mathbb{R}_\leq\cap\mathbb{P}\neq\emptyset.\label{cr3}\\
&\mathbb{R}_\leq\cap\mathbb{P}\subseteq{\mathbb{V}}_\leq\subseteq\mathbf{U(\mathbb{R}_\geq\cap\mathbb{E})},\,\mbox{if}\quad\mathbb{R}_\geq\cap\mathbb{E}\neq\emptyset.\label{cr4}
\end{align}
\end{enumerate}
\end{subequations}
\end{Proposition}
\begin{proof}
We will provide the proof of \eqref{cr1} and the other results there can be proved in a similar way. We only claim $\mathcal{R}(\mathbb{U}_\geq)\subseteq\mathbb{U}_\geq$ and refer the reader to \cite[Lemma 2.3]{MR4487984} for the remaining part. Indeed, for $X_0 \in \mathbb{U}_\geq$, it follows that $X_0 \geq X_{\mathbb{P}}$ and $X_0 \in \mathbb{R}_\geq \cap \mathbb{P}$. Thus, $X_1 = \mathcal{R} (X_0)$ is well-defined with $X_0 \geq X_1$. Furthermore, we also have
   \begin{align*}
 \mathcal{C}_{T_{X_0}}(X_0-X_\mathbb{P}) &= \mathcal{C}_{T_{X_0}}(X_0)
 -\mathcal{C}_{T_{X_0}}(X_\mathbb{P})= H_{X_0}+X_0-\mathcal{R}(X_0)-\mathcal{C}_{T_{X_0}}(X_\mathbb{P})\\
 &\geq H_{X_0}+K(X_\mathbb{T},X_0)-(H_{X_0}+X_\mathbb{P}-\mathcal{R}(X_\mathbb{P})-K(X_0, X_\mathbb{T}))\\
 &= K(X_\mathbb{T},X_0)-(X_\mathbb{P}-\mathcal{R}(X_\mathbb{P}))+K(X_0,X_\mathbb{P})
 \geq K(X_\mathbb{T},X_0),
\end{align*}
for some $X_{\mathbb{T}} \in \mathbb{T}$. Then $X_0\in \mathbb{T}$ follows immediately from Proposition~\ref{pro4} since
  $\rho(\widehat{T}_{X_{\mathbb{T}}})<1$. From~\eqref{Req-b} we have
\begin{align*}
 X_0-\mathcal{R}(X_0)&=\mathcal{C}_{T_{X_0}}(X_0)-H_{X_0}.
\end{align*}
 We see that
\begin{align*}
\mathcal{C}_{T_{X_0}}(X_{1})-H_{X_0}
=\mathcal{C}_{T_{X_0}}(X_{1})+X_0-X_{1}-\mathcal{C}_{T_{X_0}}(X_{0})
=-\mathcal{C}_{T_{X_0}}(X_0-X_{1})+X_0-X_{1}=T_0^H(\overline{X}_0-\overline{X}_{1})T_0\geq 0.
\end{align*}
  Thus, it implies that $X_{1}=\mathcal{R}(X_0)\in \mathbb{S}_\geq$ because $X_0\in \mathbb{T}$.
\end{proof}
Proposition~\ref{P1} reveals very helpful light on two extremal solution $\mathbf{X_M}$ and $\mathbf{Y_m}$. Namely, it has been shown that
$\mathbb{U}_\geq\subseteq\mathbb{P}$ and $\mathbb{U}_\leq\subseteq\mathbb{E}$ are two closed and $\mathcal{R}$-invariant sets. Hence, it follows
from Lemma~\ref{11} that every sequence $\{X_k\}$ with $X_0\in \mathbb{U}_\geq$ (resp. $X_0\in \mathbb{U}_\leq$) generated by the FPI~\eqref{fpi} converges to $\mathbf{X_M}$ (resp. $\mathbf{Y_m}$). That is, this result provides an alternative verification to Lemma~\ref{thm3p1}.
\begin{Remark}
Let $X\in\mathbb{U}_\geq\cup\mathbb{U}_\leq$.
In general, $\mathcal{R}$ is order preserving on $\bigcup\limits_{n\geq 0}\mathcal{R}^{(n)}(X)$ from Proposition~\ref{P1}. Especially,
$\mathcal{R}$ is order preserving on $\mathbb{N}_n$ if $R>0$ and $H\geq0$, $\mathcal{R}$ is order reversing if $R<0$ and $H\leq0$ on $-\mathbb{N}_n$.
\end{Remark}
It is interesting to ask whether the results~\eqref{cr3} and \eqref{cr4} lead the same conclusion as the above? The answer for this problem is unsure since $\mathbb{V}_\geq$ or $\mathbb{V}_\leq$ may not be a $\mathcal{R}$-invariant set. Based on the above reasons, we shall investigate other kind of fixed-point iteration to guarantee the existence of the convergence region to two extremal solutions $\mathbf{X_m}$ and $\mathbf{Y_M}$ in the next section, which generalizes the aforementioned works.

\section{Extremal solutions: $\mathbf{X_m}$, $\mathbf{Y_M}$, $\mathbf{Z_M}$, $\mathbf{Z_m}$, $\mathbf{W_M}$ and $\mathbf{W_m}$}\label{sec3}
Before proceeding, to advance our proposal for the existence of six extremal solutions, some relevant properties about two auxiliary equations are first given as follows, which the following relationship among some subsets on $\mathrm{dom}(\mathcal{R})$ is
needed.
 \begin{Proposition}\label{pro6}
Assume that $A$ is nonsingular and $H,\overline{H}_A\in \mathrm{dom}(\mathcal{R})$.
 The following statements hold:
 \par\noindent

 \begin{enumerate}
 \item $\widehat{F}_X\equiv\bb \overline{F}_X T_{\mathcal{R}(X)}\\ F_{\mathcal{R}(X)}-F_H\eb$ and
      $\widehat{T}_X^{(D)}\equiv\widehat{A}-\widehat{B}\widehat{F}_X=\overline{T}_X T_{\mathcal{R}(X)}$ if $X$,$\mathcal{R}(X)\in\mathrm{dom}(\mathcal{R})$.
    Moreover, $\widehat{T}_X^{(D)}=\widehat{T}_X$  for any $X\in\mathbb{R}_=$.
   \item When $X\in\mathrm{dom}(\mathcal{R})$, then
   $Z=-\mathcal{R}(X)\in\mathrm{dom}({\widetilde{\mathcal{R}}})$. Conversely, when $Z\in\mathrm{dom}({\widetilde{\mathcal{R}}})$, then $X=-\overline{\widetilde{\mathcal{R}}}(Z)\in\mathrm{dom}({\mathcal{R}})$. Furthermore,
   \begin{align}\label{f1}
   \widehat{T}_X\widehat{U}_{Z}=I
   \end{align}
   for any $X\in\mathrm{dom}(\mathcal{R})$. Moreover,
   $X\in\mathbb{O}\cap\mathbb{E}$ if and only if $Z=-\mathcal{R}(X)\in\widetilde{\mathbb{T}}\cap\widetilde{\mathbb{E}}$, and $X\in\mathbb{T}\cap\mathbb{P}$ if and only if $Z=-\mathcal{R}(X)\in\widetilde{\mathbb{O}}\cap\widetilde{\mathbb{P}}$.

   \item
 When $\mathbb{R}_\leq\cap\mathbb{P}\neq\emptyset$, we have
\begin{align}\label{f2}
\mathbb{R}_\leq\cap\mathbb{P}=-\widetilde{\mathbb{R}}_\leq\cap\widetilde{\mathbb{P}}\subseteq\widehat{\mathbb{P}}\subseteq{\mathbb{P}}.
\end{align}
(resp. $\mathbb{R}_\geq\cap\mathbb{E}=-\widetilde{\mathbb{R}}_\geq\cap\widetilde{\mathbb{E}}\subseteq\widehat{\mathbb{E}}\subseteq{\mathbb{E}}$ if $\widetilde{\mathbb{R}}_\leq\cap\widetilde{\mathbb{P}}\neq\emptyset$)

    \item When $G,H\geq 0$, we have
    \begin{align}\label{f3}
 \mathcal{R}(\mathbb{R}_\geq\cap\mathbb{N}_n)=-(\widetilde{\mathbb{R}}_\geq\cap -\mathbb{N}_n)
    \end{align}
 or equivalently $\overline{\widetilde{\mathcal{R}}}(\widetilde{\mathbb{R}}_\geq\cap\mathbb{N}_n)=-(\mathbb{R}_\geq\cap -\mathbb{N}_n)$.

 \item
When $\mathbb{T}\neq\emptyset$, we have
\begin{align}\label{f4}
  \mathbb{T}=-\overline{\widetilde{\mathcal{R}}}(\widetilde{\mathbb{O}})\subseteq\widehat{\mathbb{T}}.
 \end{align}
(resp. $\mathbb{O}=-\overline{\widetilde{\mathcal{R}}}(\widetilde{\mathbb{T}})\subseteq\widehat{\mathbb{O}}$ if $\widetilde{\mathbb{O}}\neq\emptyset$)

  \item
When $\mathbb{T}\neq\emptyset$, we have
 \begin{align}\label{f6}
\mathbb{U}_\geq=-{\widetilde{\mathbb{U}}}_\geq\subseteq\widehat{\mathbb{U}}_\geq\quad\mbox{and}\quad \mathbb{U}_\leq=-{\widetilde{\mathbb{U}}}_\leq\subseteq\widehat{\mathbb{U}}_\leq.
 \end{align}
 (resp. $\mathbb{V}_\geq=-{\widetilde{\mathbb{V}}}_\geq\subseteq\widehat{\mathbb{V}}_\geq$ and
 $\mathbb{V}_\leq=-{\widetilde{\mathbb{V}}}_\leq\subseteq\widehat{\mathbb{V}}_\leq$  if $\mathbb{O}\neq\emptyset$  )
 \end{enumerate}
 \end{Proposition}
\begin{proof}\par\noindent
\begin{enumerate}
\item
Concerning the result, which is more
technical, we refer the interested reader to the reference~\cite{CHIANG202471}.
\item
With some manipulations one obtains that 
\begin{align*}
  &I+G\overline{X}-A(I-\widetilde{G} H)^{-1}\widetilde{G}A^H\overline{X}=I+(G-\widetilde{A}^{-1}\widetilde{G}A^H)\overline{X}=I,\\
&(I-\widetilde{G}H)^{-1}=A^{-1}\widetilde{A}^{-1}.
\end{align*}
Therefore, $I-\widetilde{G}\mathcal{R}(X)$ is invertible by SMWF and its inversion can be presented in the following form:
\begin{align*}
&(I-\widetilde{G}\mathcal{R}(X))^{-1}=(I-\widetilde{G}\overline{H}-\widetilde{G}A^HX(I+G\overline{X})^{-1}A)^{-1}\\
&=A^{-1}\widetilde{A}^{-1}-A^{-1}\widetilde{G}A^H\overline{X}\widetilde{A}^{-1}
=A^{-1}\widetilde{A}^{-1}-A^{-1}G\overline{X}\widetilde{A}^{-1}=A^{-1}(I+G\overline{X})\widetilde{A}^{-1}.
\end{align*}
This yields
\begin{align*}
T_X=(I+G\overline{X})^{-1}A=\widetilde{A}^{-1}(I+\widetilde{G}Z)
={U}_Z^{-1},
\end{align*}
and thus
\[
   \widehat{T}_X\widehat{U}_Z=\overline{T_X}T_X U_Z\overline{U_Z}=\overline{T_X}\,\overline{U_Z}=I.
   \]
In addition, it is easily to verify that
\[
\widetilde{R}_Z^{-1}=(R-B^HX(I+GX)^{-1}B)^{-1}=R^{-1}+R^{-1}B^HXBR^{-1}=R^{-1}R_XR^{-1}.
\]
This implies that $X\in\mathbb{P}$ (resp. $X\in\mathbb{E}$) is equivalent to $Z\in\widetilde{\mathbb{P}}$ (resp. $X\in\widetilde{\mathbb{E}}$). The proof is thus completed.

  \item
      First, the proof of the relations $$\mathbb{R}_\leq\cap\mathbb{P}\subseteq\widehat{\mathbb{P}}\subseteq{\mathbb{P}}$$  can follow immediately from the definition of $\mathbb{P}$, $\widetilde{\mathbb{P}}$ and $\widehat{\mathbb{P}}$, thus the proof is omitted. On the other hand,
       we have the following inequality:
 \begin{align}
X_A-H_A-\overline{X}&=\overline{X}((I+G\overline{X})^{-1}-I)= -\overline{X}(I+G\overline{X})^{-1}G\overline{X}\nonumber\\
  &=-\overline{X}
  B(R+B^{H}\overline{X}B)^{-1}B^{H}\overline{X}\leq 0.\label{ine1}
\end{align}
It follows from \eqref{ine1} that $X\leq {R}(X)$ and $X\in\mathbb{P}$ if and only if
\[
\mathcal{M}_X:=\bb -X_A+H_A+\overline{X} & \overline{X}B\\B^H\overline{X} & R_{{X}}\eb\geq 0.
\]
Let $L_1=\bb I_n & -B\\ 0 & I_n\eb$ and $Y=-X$. Then,
\begin{align*}
 \mathcal{N}_X:=L_1^H \mathcal{M}_X L_1=\bb -X_A+H_A+\overline{X}  & (X_A-H_A)B\\B^H(X_A-H_A) & R_{H_A-X_A}\eb\geq 0,
\end{align*}
which yields that $ R_{H_A-X_A}\geq 0$. Now notice that $\mathrm{rank}(R_{H_A-X_A})=\mathrm{rank}(\bb -B^H & I_n\eb\mathcal{M}_X\bb -B \\ I_n \eb)\geq n$ to conclude $R_{H_A-X_A}>0$.
This implies that $\widetilde{R}_Y:=\widetilde{R}+\widetilde{B}^H Y\widetilde{B}=R_{H_A-X_A}>0$, i.e., $Y\in\widetilde{\mathbb{P}}$.
On the other hand, the Schur complement of $\mathcal{N}_X$ with respect to $R_{H_A-X_A}$ is nonnegative, that is,
\[
\mathcal{N}_X/R_{H_A-X_A}:=\overline{X}-X_A+H_A-(X_A-H_A)(BR_{H_A-X_A}^{-1}B^H)(X_A-H_A)\geq 0.
\]
Obviously, the following identities hold:
\begin{subequations}\label{der1}
\begin{align}
  BR_{H_A-X_A}^{-1}B^H&=(I+G(H_A-X_A))^{-1}G,\\
 (I+G(H_A-X_A))^{-1}&=A(I-\widetilde{G}X)^{-1}\widetilde{A},\\
 \widetilde{A}^HX&=(I+H_A G)^{-1} X_A A,\\ 
 \widetilde{H}&=(I+H_AG)^{-1}H_A=H_A(I+H_AG)^{-1}.
\end{align}
\end{subequations}
Thus, the following derivation can be obtained by using~\eqref{der1}
\begin{align*}
&\mathcal{N}_X/R_{H_A-X_A}=\overline{X}+(H_A -X_A)(I-(I+G(H_A-X_A))^{-1}(G(X_A-H_A)))\\
&=\overline{X}+(H_A -X_A)(I+G(H_A-X_A))^{-1}
=\overline{X}+(H_A -X_A)A(I-\widetilde{G}X)^{-1}\widetilde{A}\\
&=\overline{X}+(H_AA-(I+H_AG)\widetilde{A}^HX)(I-\widetilde{G}X)^{-1}\widetilde{A}\\
&=\overline{X}+(H_AA(I-A^{-1}G\widetilde{A}^HX)(I-\widetilde{G}X)^{-1}
\widetilde{A})-\widetilde{A}^HX(I-\widetilde{G}X)^{-1}\widetilde{A}\\
&=\overline{X}+\widetilde{H}-\widetilde{A}^HX(I-\widetilde{G}X)^{-1}\widetilde{A}
=-\overline{Y}+\widetilde{\mathcal{R}}(Y)\geq 0,
\end{align*}
where $Y=-X$. Namely, $Y\in\widetilde{\mathbb{R}}_\leq\cap\widetilde{\mathbb{P}}$. The converse can be proved similarly.
\item
We only have to verify the first equality since the remainder is similar. Applying Proposition~\ref{pro3} we have $Z=-\mathcal{R}(X)$ if and only if $X=-\overline{\widetilde{\mathcal{R}}}(Z)$. Therefore, $X\in\mathbb{R}_\geq$ is equivalent to $Z\in\widetilde{\mathbb{R}}_\geq$ since $X-\mathcal{R}(X)=Z-\overline{\widetilde{\mathcal{R}}}(Z)$. It is easy to check that $\widetilde{G},\widetilde{H}\geq 0$ if $G,H\geq 0$. Namely, $Z\in-\mathbb{N}_n$ if $X\in\mathbb{N}_n$. We conclude  that $-\mathcal{R}(\mathbb{R}_\geq\cap\mathbb{N}_n)\subseteq\widetilde{\mathbb{R}}_\geq\cap-\mathbb{N}_n$.
Its converse is also true.

\item
We will provide the proof when $\mathbb{T}\neq\emptyset$  and the other assumption there can be proved in a similar way.
The first conclusion of \eqref{f4}  is a direct consequence of the identity  ~\eqref{f1} and the second conclusion of \eqref{f4} can follow immediately from the definition, and thus the proof is completed.
%
\item
Our main approach in what follows is to establish two interesting equalities
for the connection of $\mathbb{U}_\geq$ with $\widetilde{\mathbb{U}}_\geq$ and $\widehat{\mathbb{U}}_\geq$ to guarantee \eqref{f6} holds. First, observe that
\begin{align*}
X+Z&=X-\mathcal{R}(X)=\mathcal{C}_{T_X}(X)-H_X,\\
   &=-\overline{\widetilde{\mathcal{R}}}(Z)+Z=\mathcal{C}_{\overline{U}_Z}(Z)-\overline{\widehat{H}}_Z.
\end{align*}
Therefore,
\begin{align*}
H_X=-Z-T_X^H \overline{X} T_X,\,\widehat{H}_Z=-\overline{X}-U_Z^H Z U_Z,
\end{align*}
thus
\begin{align*}
U_Z^H H_X U_Z=-U_Z^HZU_Z-\overline{X}=\widehat{H}_Z
\end{align*}
and vice versa. We continue in this fashion obtaining the first equality
\begin{subequations}
\begin{align}\label{f8}
\mathcal{C}_{\overline{U}_Z}(Y)-\overline{\widehat{H}}_Z
=\overline{U}_Z^H(\overline{\mathcal{C}}_{T_X}(-Y)-\overline{H}_X)\overline{U}_Z
\end{align}
for any $Y\in\mathbb{H}_n$, from which the first result $\mathbb{U}_\geq=-\widetilde{\mathbb{U}}_\geq$ follows.
Next, we show that
$\widehat{H}_{\widehat{F}_X}
=\widehat{H}+\widehat{F}_X^H\widehat{R}\widehat{F}_X$.
Following the formula~\eqref{Req-a} we have
\begin{align*}
  &\mathcal{K}_{F_X}(H)-\mathcal{K}_{F_X}(0)=(H-R(H)-\mathcal{C}_{T_X}(H)+H_{T_X})-(0-R(0)-\mathcal{C}_{T_X}(0)  +H_{F_X})\\&=T_X^H\overline{H}T_X-A^H\overline{H}(I+G\overline{H})^{-1}A.
\end{align*}
Hence,
  \begin{align*}
&H_{F_X}+T_X ^H\overline{H}_{F_X}T_X=H+\mathcal{K}_F(0)+T_X ^H\overline{H}T_X+T_X ^H\overline{\mathcal{K}}_F(0)T_X \\
&=H+\mathcal{K}_F(0)+ \mathcal{K}_F(H)-\mathcal{K}_F(0)+A^H\overline{H}(I+G\overline{H})^{-1}A+T_X ^H\overline{\mathcal{K}}_F(0)T_X \\
&=H+A^H\overline{H}(I+G\overline{H})^{-1}A+T_X ^H\overline{\mathcal{K}}_F(0)T_X + \mathcal{K}_F(H)\\
&=\widehat{H}+\widehat{F}_X^H\widehat{R}\widehat{F}_X=\widehat{H}_{\widehat{F}_X}.
 \end{align*}
 We obtain that the second equality
\begin{align}\label{f9}
  \mathcal{C}_{T_X}(Y)-H_{F_X} &=\mathcal{S}_{\widehat{T}_X}(Y)-H_{\widehat{F}_X}
  -T_X^H(\overline{\mathcal{C}}_{T_X}(Y)-\overline{H}_{F_X} )T_X,
\end{align}
\end{subequations}
 from which the second result $\mathbb{U}_\geq\subseteq\widehat{\mathbb{U}}_\geq$ follows. The proof is thus
completed.

\end{enumerate}
\end{proof}

\subsection{The existence of $\mathbf{X_m}$ and $\mathbf{Y_M}$}\label{31}
With the preceding preliminary, the first main results of this paper are
summarized in the following theorem.
\begin{Theorem}\label{mainthm}
The existence of four extremal solutions of CDARE~\eqref{cdare} can be described by
the following statements:
\begin{enumerate}
\item Assume that $\mathbb{R}_{\leq}\cap\mathbb{P}\neq\emptyset$.
\begin{enumerate}
  \item[(1a)] If $\mathbb{T}\neq\emptyset$, then $\mathbf{X_M}=\max(\mathbb{R}_{=}\cap\mathbb{P}$) exists.
  \item[(1b)] If $\mathbb{O}\neq\emptyset$ and $\overline{H}_A\in\mathrm{dom}(\mathbb{R})$, then $\mathbf{X_m}=\min(\mathbb{R}_{=}\cap\mathbb{P}$) exists.
\end{enumerate}
\item Assume that $\mathbb{R}_{\geq}\cap\mathbb{E}\neq\emptyset$.
\begin{enumerate}
  \item[(2a)] If $\mathbb{O}\neq\emptyset$ and $\overline{H}_A\in\mathrm{dom}(\mathbb{R})$, then $\mathbf{Y_M}=\max(\mathbb{R}_{=}\cap\mathbb{E}$) exists.
  \item[(2b)] If $\mathbb{T}\neq\emptyset$, then $\mathbf{Y_m}=\min(\mathbb{R}_{=}\cap\mathbb{E}$) exists.
\end{enumerate}
\end{enumerate}
\end{Theorem}
\begin{proof}
Two cases (1a) and (2b) can follow immediately from Theorem~\ref{thm3p1}. For the proof of case~(1b),
combining Proposition~\ref{pro3} and Proposition~\ref{pro6} gives that $\mathbf{X_m}=\min(\mathbb{R}_{=}\cap\mathbb{P})=-\max(-\mathbb{R}_{=}\cap\mathbb{P})
=-\max(\widetilde{\mathbb{R}}_{=}\cap\widetilde{\mathbb{P}})$ exists since $\widetilde{\mathbb{R}}_{\leq}\cap\widetilde{\mathbb{P}}\neq\emptyset$ and $\widetilde{\mathbb{T}}\neq\emptyset$. The rest is very similar to what we did to the case~(2a) above.
\end{proof}
\begin{Remark}\label{rem2}
Based on the FPI~\eqref{fpi}, i.e., $X_{k+1}=\mathcal{R}(X_k)$ with initial value $X_0\in\mathbb{H}_n$ and the dual-FPI, i.e., $Y_{k+1}=\widetilde{\mathcal{R}}(X_k)$ with initial value $Y_0\in\mathbb{H}_n$; some constructive methods are established to compute four extremal solutions $\mathbf{{X}_M}$, $\mathbf{{X}_m}$, $\mathbf{{Y}_M}$ and $\mathbf{{Y}_m}$:
\begin{align*}
\mathbf{{X}_M}&=\lim\limits_{k\rightarrow\infty}\mathcal{R}^{(k)}(X_0)
=\lim\limits_{k\rightarrow\infty}X_k,\quad\forall{{X}_0}\in\mathbb{U}_\geq.  \\
\mathbf{{X}_m}&=-\lim\limits_{k\rightarrow\infty}\widetilde{\mathcal{R}}^{(k)}(X_0)
=-\lim\limits_{k\rightarrow\infty}Y_k,\quad\forall{{X}_0}\in\mathbb{V}_\geq=-\mathbb{U}_\geq.  \\
\mathbf{{Y}_M}&=-\lim\limits_{k\rightarrow\infty}\widetilde{\mathcal{R}}^{(k)}(X_0)
=-\lim\limits_{k\rightarrow\infty}Y_k,\quad\forall{{Y}_0}\in\mathbb{V}_\leq=-\mathbb{U}_\leq. \\
\mathbf{{Y}_M}&=\lim\limits_{k\rightarrow\infty}\mathcal{R}^{(k)}(X_0)
=\lim\limits_{k\rightarrow\infty}X_k,\quad\forall{{Y}_0}\in\mathbb{U}_\leq.  \end{align*}
We notice that $Y_{k+1}=\widetilde{\mathcal{R}}(Y_k)=-\mathcal{R}^{-1}(-Y_k)$.
\end{Remark}
\begin{Proposition}\label{pro33}
\par\noindent
Assume that $A$ is nonsingular. When $H,\overline{H}_A\in \mathrm{dom}(\mathcal{R})$ some useful statements are given.
\begin{enumerate}

  \item When $\mathbb{R}_\leq\cap\mathbb{P}\neq\emptyset$, then $\mathbf{X_M}=\mathbf{\widehat{X}_M}=-\mathbf{\widetilde{X}_m}$ if $\mathbb{T}\neq\emptyset$, and $\mathbf{Y_m}=\mathbf{\widehat{Y}_m}=-\mathbf{\widetilde{Y}_M}$ if $\mathbb{O}\neq\emptyset$.
  \item When $\mathbb{R}_\geq\cap\mathbb{E}\neq\emptyset$, then $\mathbf{Y_M}=\mathbf{\widehat{Y}_M}=-\mathbf{\widetilde{Y}_m}$ if $\mathbb{T}\neq\emptyset$, and $\mathbf{X_m}=\mathbf{\widehat{X}_m}=-\mathbf{\widetilde{X}_M}$ if $\mathbb{O}\neq\emptyset$.
\end{enumerate}
\end{Proposition}
\begin{proof}
We will prove the part~2, and the other part there can be proved in a similar way. Assume that $\mathbb{R}_\geq\cap\mathbb{E}\neq\emptyset$ and $\mathbb{T}\neq\emptyset$. By using Proposition~\ref{pro6}, recalling $\widetilde{\mathbb{R}}_\geq\cap\widetilde{\mathbb{E}}=-\mathbb{R}_\geq\cap\mathbb{E}\neq\emptyset$
and $\widetilde{\mathbb{O}}\neq\emptyset$, and the conditions that guarantee the existence of $\widetilde{\mathbf{Y}}_m$. In particular, we have deduced that
\[
{\mathbf{Y_M}}=\max(\mathbb{R}_=\cap\mathbb{E})=-\min(\mathbb{D}_=\cap\mathbb{E})=\mathbf{\widetilde{Y}_m}.
\]
In addition, $\mathbf{\widehat{Y}_M}$ exists since $\widehat{\mathbb{R}}_=\cap\widehat{\mathbb{E}}\neq\emptyset$
and $\widehat{\mathbb{O}}\neq\emptyset$. Moreover, for any $X_0\in\mathbb{U}_\leq\subseteq\widehat{\mathbb{U}}_\leq$, we obtain
\[
\mathbf{{Y}_M}=\lim\limits_{k\rightarrow\infty}\mathcal{R}^{(k)}(X_0)
=\mathcal{R}^{(2k)}(X_0)=\widehat{\mathcal{R}}^{(k)}(X_0)=\mathbf{\widehat{Y}_M}.
\]

%
%

\end{proof}
Theorem~\ref{mainthm} just provides two sufficient conditions for the existence of $\mathbf{X_m}$ and $\mathbf{Y_M}$. It is interesting to point out that there exist other sufficient condition to guarantee the existence. To this end, we first give the following definition to provide a relatively sufficient condition.
\subsection{The existence of $\mathbf{Z_M}$,  $\mathbf{Z_m}$, $\mathbf{W_M}$ and  $\mathbf{W_m}$}
In the previous Subsection~\ref{31}, Theorem~\ref{mainthm} just provides two sufficient conditions for the existence of $\mathbf{X_m}$ and $\mathbf{Y_M}$. It is interesting to point out that there exist similar sufficient condition to guarantee the existence~$\mathbf{Z_M}$ and $\mathbf{Z_m}$. To this end, we first give the following result to provide a relatively sufficient condition.
\begin{Theorem}\label{thm0}
Assume that $\mathcal{R}$ is order preserving on  $\mathbb{N}_n$.
Then, $\mathbf{Z_M}$ exists if $\mathbf{U}(\mathbb{R}_\leq\cap\mathbb{N}_n)\cap
(\mathbb{R}_\geq\cap\mathbb{N}_n)\neq\emptyset$. In contrast $\mathbf{Z_m}$ exists if $\mathbf{L}(\mathbb{R}_\geq\cap\mathbb{N}_n)\cap
(\mathbb{R}_\leq\cap\mathbb{N}_n)\neq\emptyset$.
\end{Theorem}
\begin{proof}
The result follows immediately based on Lemma~\ref{11}, or another illustration see \cite[Theorem 2.1]{c22} for more detail.
\end{proof}
In what follows, the second main
results concerning the sufficient conditions for the existence of maximum/minimum positive and negative semidefinite solutions of CDARE~\eqref{cdare} are stated.
\begin{Theorem}\label{thm1}
\par\noindent
Suppose that $H\geq 0$.
Then, the following statements hold. 
\begin{enumerate}
  \item 
      $\mathbf{Z_m}=\min(\mathbb{R}_{=}\cap\mathbb{N}_n)$ exists if $\mathcal{R}$ is order preserving on  $\mathbb{N}_n$ and $\mathbb{R}_{\geq}\cap\mathbb{N}_n\neq\emptyset$.
\item  $\mathbf{Z_M}=\max(\mathbb{R}_{=}\cap \mathbb{N}_n)$ exists if $\mathbb{T}\neq\emptyset$ and $R>0$.
      \item 
      $\mathbf{W_m}=\min(\mathbb{R}_{=}\cap-\mathbb{N}_n)$ exists if $\mathbb{O}\neq\emptyset$ and $R>0$.
\item  $\mathbf{W_M}=\max(\mathbb{R}_{=}\cap -\mathbb{N}_n)$ exists if $\widetilde{\mathcal{R}}$ is order preserving on  $\mathbb{N}_n$ and $\mathbb{R}_{\geq}\cap -\mathbb{N}_n\neq\emptyset$.

\end{enumerate}
\end{Theorem}
\begin{proof}
\begin{enumerate}
  \item We claim that
\[
[0,H]\subseteq  \mathbf{L}(\mathbb{R}_\geq\cap\mathbb{N}_n)\cap
(\mathbb{R}_\leq\cap\mathbb{N}_n).
\]
Indeed, when $X_\ast\in[0,H]$ we see that $\mathcal{R}(X_\ast)\geq\mathcal{R}(0)=H\geq X_\ast$ , and $X\geq\mathcal{R}(X)\geq\mathcal{R}(0)=H\geq X_\ast$ for all $X\in\mathbb{R}_\geq\cap\mathbb{N}_n$. The result is a consequence of the Theorem~\ref{thm0}.

  \item
We claim that
\[
\mathbb{U}_\geq\subseteq  \mathbf{U}(\mathbb{R}_\leq\cap\mathbb{N}_n)\cap
(\mathbb{R}_\geq\cap\mathbb{N}_n).
\]
Indeed, $\mathcal{C}_{T_{X_\mathbb{T}}}(X)\geq H_{X_\mathbb{T}}\geq 0$ implies that $X\geq 0$,
where $X_\mathbb{T}\in\mathbb{T}$. We conclude that
$\mathbb{U}_\geq\subseteq\mathbb{N}_n$ and thus $\mathbb{U}_\geq\subseteq\mathbb{R}_\geq$. Next,
 when $X_\ast\in\mathbb{U}_\geq$ and $X\in\mathbb{R}_\leq\cap\mathbb{N}_n$ we see that $\mathcal{C}_{T_{X_\mathbb{T}}}(X_\ast)\geq H_{X_\mathbb{T}}$ for some $X_\mathbb{T}\in\mathbb{T}$, and $H_{X_{\mathbb{T}}}\geq -(X-\mathcal{R}(X))+K(X_\mathbb{T},X)+\mathcal{C}_{T_{X_\mathbb{T}}}(X)\geq\mathcal{C}_{T_{X_\mathbb{T}}}(X)$.
In other words, we have $\mathcal{C}_{T_{X_\mathbb{T}}}(X_\ast)\geq \mathcal{C}_{T_{X_\mathbb{T}}}(X)$ and thus $\mathbb{U}_\geq\subseteq  \mathbf{U}(\mathbb{R}_\leq\cap\mathbb{N}_n)$. This fact together with Theorem~\ref{thm0} leads to the existence for $\mathbf{Z_M}$.
\item The
results can follow immediately from $\mathbf{W_m}=\min(\mathbb{R}_{=}\cap-\mathbb{N}_n)
    \equiv-\max(\widetilde{\mathbb{R}}_{=}\cap\mathbb{N}_n)=-\widetilde{\mathbf{Z}}_M$ exists, if $\widetilde{\mathbb{T}}=\mathbb{O}\neq\emptyset$.
\item
Observe that
$\mathbf{W_M}=\max(\mathbb{R}_{=}\cap-\mathbb{N}_n)
    =-\min(\widetilde{\mathbb{R}}_{=}\cap\mathbb{N}_n)=-\widetilde{\mathbf{Z}}_m$ and $\widetilde{\mathbb{R}}_\geq\cap\mathbb{N}_n=-(\mathbb{R}_\geq\cap-\mathbb{N}_n)$. The proof is completed by applying part~1.
\end{enumerate}

\end{proof}
\begin{Remark}\label{rem3}
In analogy with Remark~\ref{rem2}, four extremal solutions $\mathbf{{Z}_M}$, $\mathbf{{Z}_m}$, $\mathbf{{W}_M}$ and $\mathbf{{W}_m}$ can be computed using the following iterations:
\begin{align*}
\mathbf{{Z}_M}&=\lim\limits_{k\rightarrow\infty}\mathcal{R}^{(k)}(X_0)
=\lim\limits_{k\rightarrow\infty}X_k,\quad\forall{{X}_0}\in\mathbb{U}_\geq.  \\
\mathbf{{Z}_m}&=\lim\limits_{k\rightarrow\infty}\mathcal{R}^{(k)}(X_0)
=\lim\limits_{k\rightarrow\infty}X_k,\quad\forall{{X}_0}\in[0,H].  \\
\mathbf{{W}_M}&=-\lim\limits_{k\rightarrow\infty}\widetilde{\mathcal{R}}^{(k)}(X_0)
=-\lim\limits_{k\rightarrow\infty}Y_k,\quad\forall{{Y}_0}\in[0,\widetilde{H}]. \\
\mathbf{{W}_m}&=-\lim\limits_{k\rightarrow\infty}\widetilde{\mathcal{R}}^{(k)}(X_0)
=-\lim\limits_{k\rightarrow\infty}Y_k,\quad\forall{{Y}_0}\in-\mathbb{V}_\geq.  \end{align*}
\end{Remark}
\begin{Lemma}\label{lem11}
Let $C=A\oplus B\in\mathbb{C}^{(n+m)\times (n+m)}$ with $A\in\mathbb{C}^{n\times n}$ and $B\in\mathbb{C}^{m\times m}$. Then,
\begin{enumerate}
  \item $A\geq 0$ if $C\geq U^H B U$ for some $U\in\mathbb{C}^{m\times (n+m)}$.
  \item $B\leq 0$ if $C\leq V^H A V$ for some $V\in\mathbb{C}^{n\times (n+m)}$.
\end{enumerate}

\end{Lemma}
\begin{proof}
We give only the proof of the first part, the same proof remains valid for second part. For the sake of simplicity, we partition $U$ as $U=\bb U_1 & U_2\eb$ with $U_1\in\mathbb{C}^{m\times n}$ and $U_2\in\mathbb{C}^{m\times m}$. If there exist nonzero vector $x_1\in\mathbb{C}^{n}$ so that $x_1^H A x_1< 0$. Choose a scalar $\lambda\not\in\sigma(U_2)$ and $|\lambda|=1$. Let $x=\bb x_1 \\ x_2\eb$ with $x_2=(\lambda I-U_2)^{-1}U_1x_1$. Thus, $Ux=U_1x_1+U_2x_2=(I+U_2(\lambda I-U_2)^{-1})U_1x_1=\lambda x_2$.
It follows from a direct computation that
\begin{align*}
x_1^H A x_1+x_2^H B x_2 =x^H (A\oplus B)x\geq (Ux)^H B(Ux)=x_2^H B x_2,
\end{align*}
which yields $x_1^H A x_1\geq 0$ and a contradiction occurs.
\end{proof}
\begin{Corollary}\label{pro1}
\par\noindent
 When $R=\widehat{R}_1^H\widehat{R}_1>0$ and $H\geq 0$, then $\mathbb{R}_{\leq}\cap\mathbb{P}\neq\emptyset$ and $\mathbb{R}_{\geq}\subseteq\mathbb{P}$. Moreover, we have
       \begin{enumerate}
         \item $\mathbf{Z_m}$ exists if $\mathbb{R}_{\geq}\neq\emptyset$.
         \item When $A$ is nonsingular, $\mathbf{W_M}$ exists if $\mathbb{R}_{\geq}\cap -\mathbb{N}_n\neq\emptyset$.
         \item When $H>0$ and $A$ is nonsingular, $\mathbf{Z_M}=\mathbf{X_M}$ exists and $\mathbf{W_m}=\mathbf{X_m}$ exists. Moreover,
          \[
\mathbf{X_m}=\mathbf{W_m}\leq\mathbf{W_M}\leq 0\leq \mathbf{Z_m}\leq \mathbf{Z_M}= \mathbf{X_M}
          \]
       \end{enumerate}
        $(\mathbf{Y_M},\mathbf{Y_m})$ doesn't exist in above either case.
\end{Corollary}
\begin{proof}
The first result follows immediately from $0\in\mathbb{R}_{\leq}\cap\mathbb{P}$.
Notice that $X\in \mathbb{R}_{\geq}$ is equivalent to $X\geq H+T_X^H(\overline{X}+\overline{X}G\overline{X})T_X$. A direct computation yields
\begin{align*}
  \bb R & -B^H \overline{X} T_X\\-T_X^H \overline{X} B & X-T_X^H \overline{X} T_X\eb &\geq \bb R & -B^H \overline{X} T_X\\-T_X^H \overline{X} B^H & T_X^H \overline{X}G\overline{X} T_X\eb
  =\bb {R}_1^H \\ -T_X^H \overline{X}{B}_1\eb \bb {R}_1 & -{B}_1^H \overline{X} T_X \eb \geq 0,
\end{align*}
where ${B}_1=B{R}_1^{-1}$. Therefore,
\begin{align*}
  \bb R+B^H \overline{X} B & 0\\0 & X\eb &\geq \bb B^H \overline{X}B  & B^H \overline{X} T_X\\T_X^H \overline{X} B & T_X^H \overline{X} T_X\eb=\bb B^H \\ T_X^H \eb \overline{X} \bb B & T_X \eb.
\end{align*}
From $0\not\in\sigma(R+B^H \overline{X} B)$ and Lemma~\ref{lem11} we conclude that $R+B^H \overline{X} B>0$ or $X\in\mathbb{P}$.
\begin{enumerate}
  \item
  From the fact
  $\mathbb{R}_\geq\cap\mathbb{P}=\mathbb{R}_\geq\neq\emptyset$ and
  $\mathcal{R}$ is order preserving on $\mathbb{N}_n$ if $G,H\geq 0$. The existence of $\mathbf{Z_m}$ is guaranteed by Theorem~\ref{thm1}.
  \item One can check that $\widetilde{G},\widetilde{H}\geq 0$ when $R>0$ and $H\geq 0$. Hence, $\widetilde{\mathcal{R}}$ is order preserving on  $\mathbb{N}_n$ and $\mathbb{R}_{\geq}\cap -\mathbb{N}_n=-(\widetilde{\mathbb{R}}_{\geq}\cap \mathbb{N}_n)\neq\emptyset$.
      The proof is thus completed.


  \item The first conclusion obtained from Proposition~\ref{pro2} and Theorem~\ref{thm1}. Moreover, $\mathbf{X_M}=\mathbf{Z_M}=\lim\limits_{k\rightarrow\infty} X_k$ with any $X_0\in\mathbb{U}_\geq$ and $\mathbf{X_m}=-\mathbf{\widetilde{X}_M}=-\mathbf{\widetilde{Z}_M}=\mathbf{Z_m}$.
       The remaining part immediately follows from the fact $\mathbb{R}_\geq\cap\mathbb{E}=\emptyset$.
\end{enumerate}

\end{proof}
\section{The definiteness of the matrix $R_X$} \label{sec4}
Our primary goal in this section is to study whether four extremal solutions $\mathbf{X_M}$, $\mathbf{X_m}$, $\mathbf{Y_M}$ and $\mathbf{Y_m}$ exist simultaneously. An interesting result is first provided to an useful case of coefficient matrices associated with the positive definite matrix $R$ and the positive semidefinite matrix $H$. In the more general case in which the nonsingular matrix $R$ and the Hermitian matrix $H$, CDARE~\eqref{cdare} can be transformed into the transformed-DARE~\eqref{dare2-a} if $R_H$ is nonsingular. An application of this DARE yields that $R_X>0$ or $R_X<0$ for all $X\in\mathbb{R}_=$ under mild assumption, as we will discuss later. With the previous lemma as preliminary, the first result can be
given.

\begin{Lemma}\label{pro1111}
 If $\widehat{\mathbb{R}}_=\cap\widehat{\mathbb{P}}\neq\phi$, then $\widehat{\mathbb{R}}_=\subseteq\widehat{\mathbb{P}}$.
\end{Lemma}
\begin{proof}
 Let the resolvent of $\widehat{A}$, the matrix function $\varphi_{\widehat{A}}: \mathbb{D}_=\backslash\sigma(\widehat{A})\rightarrow\mathbb{C}^n$ be defined by $\varphi_{\widehat{A}}(\lambda)=(\lambda I-\widehat{A})^{-1}$ for all $\lambda\in \mathbb{D}_=\backslash\sigma(\widehat{A})$, where $\mathbb{D}_{=}$ denotes the unit disk. Let $\lambda\in \mathbb{D}_=\backslash\sigma(\widehat{A})$. Note that $I+\widehat{A}\varphi_{\widehat{A}}(\lambda)=\lambda\varphi_{\widehat{A}}(\lambda)$. Thus,
$(I+\varphi_{\widehat{A}}(\lambda)^H\widehat{A}^H)X(I+\widehat{A}\varphi_{\widehat{A}}(\lambda))=\bar{\lambda}\varphi_{\widehat{A}}(\lambda)^H X \lambda\varphi_{\widehat{A}}(\lambda)=\varphi_{\widehat{A}}(\lambda)^H X \varphi_{\widehat{A}}(\lambda)$ and
\begin{align}\label{1}
&\varphi_{\widehat{A}}(\lambda)^H \mathcal{S}_{\widehat{A}}(X) \varphi_{\widehat{A}}(\lambda)=X+\varphi_{\widehat{A}}(\lambda)^H\widehat{A}^H X+X\widehat{A}\varphi_{\widehat{A}}(\lambda).
\end{align}
 The matrix $\widehat{R}_X$ is nonsingular if $X\in\widehat{\mathcal{R}}_=$. For the sake of simplicity, the matrix operator $\phi_X:\mathbb{D}_=\backslash\sigma(\widehat{A})\rightarrow\mathbb{C}^n$ be defined by $\phi_X(\lambda):=I+\widehat{F}_X\varphi_{\widehat{A}}(\lambda)\widehat{B}=I+\widehat{R}_X^{-1}\widehat{B}^HX\widehat{A}\varphi_{\widehat{A}}(\lambda)\widehat{B}=\widehat{R}_X^{-1}\widehat{R}_{\lambda X\varphi_{\widehat{A}}(\lambda)}$.  Note that
\[
\widehat{H}=X-\widehat{A}^HX\widehat{A}+\widehat{F}_X^H \widehat{R}_X \widehat{F}_X
\]
if $X\in\widehat{\mathcal{R}}_=$. Observe that
\begin{align}
\varphi_{\widehat{A}}(\lambda)^H \widehat{H} \varphi_{\widehat{A}}(\lambda)&=\varphi_{\widehat{A}}(\lambda)^H (X-\widehat{A}^HX\widehat{A}+\widehat{F}_X^H \widehat{R}_X \widehat{F}_X) \varphi_{\widehat{A}}(\lambda)\nonumber\\
&=X+\varphi_{\widehat{A}}(\lambda)^H\widehat{A}^H X+X\widehat{A}\varphi_{\widehat{A}}(\lambda)+\varphi_{\widehat{A}}(\lambda)^H \widehat{F}_X^H \widehat{R}_X \widehat{F}_X\varphi_{\widehat{A}}(\lambda).\label{222}
\end{align}
A direct calculation together with~\eqref{222} show that the Popov function~\cite[Theorem 2.5.10.]{MR1997753} associated with~\eqref{cdare-a}
\[
\Psi(\lambda):=\begin{pmatrix}
                 \widehat{B}^H\varphi_{\widehat{A}}(\lambda)^H & I
               \end{pmatrix}
               \begin{pmatrix}
                 \widehat{H} & 0 \\
                 0 & \widehat{R}
               \end{pmatrix}
               \begin{pmatrix}
                 \varphi_{\widehat{A}}(\lambda)\widehat{B} \\ I
               \end{pmatrix}\equiv \widehat{R}_{\varphi_{\widehat{A}}(\lambda)^H \widehat{H}\varphi_{\widehat{A}}(\lambda)}
\]
can be rewritten as
\begin{align}\label{2}
 &\Psi(\lambda)=\widehat{R}_X+\widehat{B}^H\left(\varphi_{\widehat{A}}(\lambda)^H\widehat{A}^H X+X\widehat{A}\varphi_{\widehat{A}}(\lambda)+\varphi_{\widehat{A}}(\lambda)^H \widehat{F}_X^H \widehat{R}_X \widehat{F}_X\varphi_{\widehat{A}}(\lambda)\right)\widehat{B}\nonumber\\
 &=\widehat{R}_X+\widehat{B}^H\varphi_{\widehat{A}}(\lambda)^H\widehat{F}_X^H\widehat{R}_X+\widehat{R}_X\widehat{F}_X\varphi_{\widehat{A}}(\lambda)\widehat{B}
 +\widehat{B}^H\varphi_{\widehat{A}}(\lambda)^H \widehat{F}_X^H \widehat{R}_X \widehat{F}_X\varphi_{\widehat{A}}(\lambda)\widehat{B}\nonumber\\
  &=(I+\widehat{B}^H\varphi_{\widehat{A}}(\lambda)^HF_X^H)\widehat{R}_X(I+F_X\varphi_{\widehat{A}}(\lambda)\widehat{B})=\phi_X(\lambda)^H \widehat{R}_X \phi_X(\lambda),
  \end{align}
 here we using $\widehat{R}_X\widehat{F}_X=\widehat{B}^H X \widehat{A}$.
With the help of SMWF, we conclude that $\phi_X(\lambda)=I+F_X\varphi_{\widehat{A}}(\lambda)\widehat{B}$ is nonsingular if and only if $\varphi_{\widehat{A}}(\lambda)^{-1}+\widehat{B}\widehat{R}_X^{-1}\widehat{B}^H X\widehat{A}=\lambda I-\widehat{A}+\widehat{B}\widehat{F}_X=\lambda I-D_X$ is nonsingular. Therefore, $\mathcal{Z}_X(\phi):=\{\lambda\in \mathbb{C}\, |\, \det(\phi_X(\lambda))=0\}\equiv\sigma(D_X)$ is a finite set of $\mathbb{C}$.

For any $W_1,W_2\in\widehat{\mathbb{R}}_=$, we choose $\lambda_0\in\mathbb{D}_=\backslash(\sigma(\widehat{A})\cup \sigma(D_{W_1})\cup \sigma(D_{W_2}))$ and we see that
$ \widehat{R}_{W_i} \sim \widehat{R}_{\varphi_{\widehat{A}}(\lambda_0)^H H\varphi_{\widehat{A}}(\lambda_0)}$ for any $W_i\in\widehat{\mathbb{R}}_=$, $i=1,2$, where the notion ``$\sim$'' to denote the congruence equivalent relation. The proof is complete.

\end{proof}
We are now in a position to show that
${\mathbb{R}}_=\cap \mathbb{P}\neq\emptyset$ and ${\mathbb{R}}_=\cap \mathbb{E}\neq\emptyset$ are mutually exclusive, which are the third main results.
\begin{Theorem}\label{1111}
Assume that $H\in\mathrm{dom}(\mathcal{R})$.
 We have ${\mathbb{R}}_=\subseteq{\mathbb{P}}$ if ${\mathbb{R}}_=\cap{\mathbb{P}}\neq\emptyset$ and ${\mathbb{R}}_=\subseteq{\mathbb{E}}$ if ${\mathbb{R}}_=\cap{\mathbb{E}}\neq\emptyset$.
In other words, $\mathbf{X_M}$ and $\mathbf{X_m}$ both doesn't exist if either $\mathbf{Y_M}$ or $\mathbf{Y_m}$ exist and vice versa.
\end{Theorem}
\begin{proof}

If $X\in\mathbb{R}_=$ with $R_X>0$, the positiveness of $\widehat{R}_X$ can be verified directly by going through the
 following computation
 \begin{align*}
 &\widehat{R}_X
 =\bb  \overline{R_X} &\overline{B}^H X\overline{A}B\\ B^H \overline{A}^H X\overline{B} & R+B^H(\overline{X}+\overline{A}^H X\overline{B}(\overline{R_X})^{-1}\overline{B}^H
 X\overline{A}) B\eb\\
 &=0\oplus \overline{R_X}+\bb \overline{R_X} \\ B^H\overline{A}^HX\overline{B}  \eb
 (\overline{R_X})^{-1}\bb  \overline{R_X} &
 \overline{B}^H X \overline{A} B
  \eb>0.
 \end{align*}
Conversely, it is clear that $R_X>0$ if  $\widehat{R}_X>0$.
Namely, $\widehat{\mathbb{P}}\subseteq\mathbb{P}$ and $\mathbb{R}_{=}\cap\mathbb{P}\subseteq\widehat{\mathbb{P}}$.
It is obtained from~Lemma~\ref{pro1111} that
$\mathbb{D}_=\subseteq\widehat{\mathbb{P}}$ if $\mathbb{D}_=\cap\widehat{\mathbb{P}}\neq\emptyset$. Concerning the first part of the theorem, we have $\mathbb{R}_=\cap \mathbb{P}\subseteq \mathbb{D}_=\cap \widehat{\mathbb{P}}$, thus $\mathbb{D}_=\cap \widehat{\mathbb{P}}\neq\emptyset$ and
\[
{\mathbb{R}}_=\subseteq\mathbb{D}_=\subseteq\widehat{\mathbb{P}}\subseteq{\mathbb{P}}.
\]
The same proof remains valid for $\mathbb{R}_=\subseteq\mathbb{E}$ if $\mathbb{R}_=\cap\mathbb{E}\neq\emptyset$ and thus the detail is omitted.
\end{proof}
\section{(Almost) stabilizing and anti-stabilizing  solutions}\label{sec4.5}
As we have seen in Theorem~\ref{thm3p1} previously, the spectral radius of $\widehat{T}_{\mathbf{X_M}}$ is always less than or equal to 1 if it exist. It is interesting to ask whether the converse is also true? Namely, the Hermitian solution $X$ of CDARE~\eqref{cdare} with $\rho(\widehat{T}_{X})\leq 1$ whether coincide with $\mathbf{X_M}$?
For this reason, the section provides a detailed exposition of the relationship between extremal solutions and a specific solutions $X$ with $\rho(\widehat{T}_X)\leq 1$ or $\mu(\widehat{T}_X)\geq 1$. Some alternating sufficient conditions for the existence of extremal solutions are first established.
Moreover, we also provide some equivalent conditions for the assumption of the existence from a different point of view.
\begin{Definition}
\par\noindent
\begin{enumerate}
  \item $\mathbf{X_s}$ is called the (almost) stabilizing solution of CDARE~\eqref{cdare} if $\mathbf{X_s}\in\mathbb{R}_=$ and $\rho(\widehat{T}_{\mathbf{X_s}})< 1$. $(\rho(\widehat{T}_{\mathbf{X_s}})\leq 1)$
  \item $\mathbf{X_{a.s}}$ is called the (almost) anti-stabilizing solution of CDARE~\eqref{cdare} if $\mathbf{X_{a.s}}\in\mathbb{R}_=$ and $\mu(\widehat{T}_{\mathbf{X_{a.s}}})> 1$. $(\mu(\widehat{T}_{\mathbf{X_{a.s}}})\geq 1)$
\end{enumerate}
\end{Definition}
\begin{Definition}~\label{defsta}
\par\noindent
We call that the matrix pair $A-\lambda B\in\mathbb{C}^{n\times n}\times\mathbb{C}^{n\times m}$ is stabilizable (anti-stabilizable) if $\rho(A-BF)<1$ ($\mu(A-BF)>1$)
for some $F\in\mathbb{F}^{m\times m}$.
\end{Definition}
\begin{Remark}
A well-known equivalent definition on the stabilizability (resp. anti-stabilizability) of matrix pair $(A,B)$ is $\mathrm{rank} [A - \lambda I\ B] = n$ for all $\lambda \in \widehat{\mathbb{R}}_\geq$(resp. $\mathbb{D}_\leq$).
\end{Remark}
The relationship between almost (anti-)stabilizable pairs and some subsets of $\mathrm{dom}({\mathcal{R}})$  is characterized in the following lemma, which will be used in the proof of main result later on.
\begin{Lemma}\label{thm3p34}
Assume that $H\in\mathrm{dom}({\mathcal{R}})$. Then,
\begin{enumerate}
  \item
  \begin{enumerate}
    \item
    When $\widehat{\mathbb{R}}_{\leq}\cap\widehat{\mathbb{P}}\neq\emptyset$  , the pair $(\widehat{A},\widehat{B})$ is stabilizable if and only if $\widehat{\mathbb{R}}_\geq\cap\widehat{\mathbb{T}}\cap\widehat{\mathbb{P}}\neq\color{black}\emptyset$
    if and only if $\widehat{\mathbb{T}}\neq\color{black}\emptyset$.
    \item
     When $\widehat{\mathbb{R}}_{\geq}\cap\widehat{\mathbb{E}}\neq\emptyset$,
     the pair $(\widehat{A},\widehat{B})$ is stabilizable if and only if  $\widehat{\mathbb{R}}_\leq\cap\widehat{\mathbb{T}}\cap\widehat{\mathbb{E}}\neq\color{black}\emptyset$
     if and only if $\widehat{\mathbb{T}}\neq\color{black}\emptyset$.
  \end{enumerate}

%
   \item
    \begin{enumerate}
      \item  When $\widehat{\mathbb{R}}_{\leq}\cap\widehat{\mathbb{P}}\neq\emptyset$,
            the pair $(\widehat{A},\widehat{B})$ is anti-stabilizable
            if and only if  $\widehat{\mathbb{O}}\neq\color{black}\emptyset$.
      \item  When $\widehat{\mathbb{R}}_{\geq}\cap\widehat{\mathbb{E}}\neq\emptyset$,
           the pair $(\widehat{A},\widehat{B})$ is anti-stabilizable if and only if  $\widehat{\mathbb{O}}\neq\color{black}\emptyset$.
    \end{enumerate}



\end{enumerate}
\end{Lemma}
\begin{proof}
\begin{enumerate}
  \item
  We only need to show the part~(b) and the proof of the remaining part is similar.
  Assume that $\widehat{\mathbb{R}}_\geq\cap\widehat{\mathbb{E}}\neq\emptyset$.
  {\color{black}Let $F\in \mathbb{C}^{n\times m}$ with $\rho({A_F})<1$ and $X_\star = \mathcal{S}_{A_F}^{-1} (Z)$ by taking an arbitrary matrix $Z$ with $Z\leq H_F$. We shall now prove that $X_\star\in\widehat{\mathbb{R}}_\leq\cap\widehat{\mathbb{E}}\cap\widehat{\mathbb{T}}$.}
First, from \eqref{Req-a} we have
 \begin{align*}
 X_{\mathbb{E}}-\widehat{\mathcal{R}}(X_{\mathbb{E}})&=\mathcal{S}_{A_{F}}(X_{\mathbb{E}})-\widehat{H}_F+\widehat{\mathcal{K}}_F(X_{\mathbb{E}})\geq 0,
 \end{align*}
for any $X_{\mathbb{E}}\in\widehat{\mathbb{R}}_\geq \cap \widehat{\mathbb{E}}$.
Therefore, $\mathcal{S}_{A_{F}}(X_{\mathbb{E}})\geq \widehat{H}_F\geq Z = \mathcal{S}_{A_{F}}(X_{\star})$, and thus $X_{\mathbb{E}}\geq X_{\star}$. That is, $X_{\star}\in\widehat{\mathbb{E}}\subseteq \mathrm{Dom}({\widehat{\mathcal{R}}})$.
On the other hand, we also have
\begin{align*}
X_\star-\mathcal{R}(X_\star)&=\mathcal{S}_{A_{F}}(X_\star)-\widehat{H}_F+\widehat{\mathcal{K}}_F(X_\star)\leq \widehat{\mathcal{K}}_F(X_\star)\leq 0.
 \end{align*}
Namely, $X_\star\in\widehat{\mathbb{R}}_\leq$. Next, applying Lemma~\ref{lemma23} we obtain
\begin{align*}
&X_\star-\widehat{\mathcal{R}}(X_\star) =\mathcal{S}_{T_{X_\star}}(X_\star)-\widehat{H}_{X_\star},\\
X_{\mathbb{E}}-\widehat{\mathcal{R}}(X_{\mathbb{E}}) &=\mathcal{S}_{T_{X_\star}}(X_{\mathbb{E}})-\widehat{H}_{F_{X_\star}} + \color{black}\widehat{\mathcal{K}}_{F_{X_\star}} (X_{\mathbb{E}}).
 \end{align*}
From which we deduce that
\begin{subequations}
\begin{align}
  &\mathcal{C}_{T_{X_\star}}(X_\star-X_{\mathbb{E}})
=X_\star-\mathcal{R}(X_\star)+H_{F_{X_\star}}
  -(X_{\mathbb{E}}-\mathcal{R}(X_{\mathbb{E}}) +H_{F_{X_\star}}\label{333a}\\
  &- {\color{black}\mathcal{K}_{F_{X_\star}} (X_{\mathbb{E}}))} \leq {\color{black}\mathcal{K}_{F_{X_\star}} (X_{\mathbb{E}})} -(X_{\mathbb{E}}-\mathcal{R}(X_{\mathbb{E}}))+\mathcal{K}_F(X_\star) \leq \mathcal{K}_F(X_\star)\leq 0.\label{333b} 
\end{align}
\end{subequations}
It follows from Lemma that $\rho(\widehat{T}_{X_\star})< 1$ and hence $X_\star\in\mathbb{T}\cap\mathbb{E}$.
   {\color{black}If $X_\star \in \mathbb{T} \cap \mathbb{P}$, then $F:=F_{X_\star}\in \mathbb{F}$ because
  $\rho (\widehat{A}_F) = \rho (\widehat{T}_{X_\star}) < 1$.}
  {\color{black}Since $\mathbb{R}_{\leq}\cap\mathbb{P}\neq\color{black}\emptyset$ and the statement (ii) holds, it follows from Theorem that
$\mathbf{X_M} \in \mathbb{R}_= \cap \mathbb{P}$. Moreover, the remaining parts of the proof are listed below.}
  \item By a similar argument of dual-CDARE associated with original CDARE, the dual-DARE associated with DARE is introduced as following:
\begin{equation}\label{ddare}
  {X}= {\widehat{\mathcal{R}}_d}(X)  :=  \widehat{A}_d^H {X} \widehat{A}_d - \widehat{A}_d^H {X}\widehat{B}_d
(\widehat{R}_d + \widehat{B}_d^H {X} \widehat{B}_d)^{-1}\widehat{B}_d^H {X} {A}_d + \widehat{H}_d,
\end{equation}
where the coefficient matrices are {\color{black} given} by
\begin{align}\label{tldA}
\widehat{A}_d &= \widehat{A}^{-1}-\widehat{B}_d\widehat{R}_d^{-1}\widehat{B}^H\widehat{H}^{(A)},\, \widehat{B}_d = \widehat{A}^{-1}\widehat{B},\,\widehat{R}_d = \widehat{R} + \widehat{B}^H\widehat{H}^{(A)} \widehat{B},\\
\widehat{H}_d&=(I+\widehat{H}^{(A)}\widehat{G})^{-1}\widehat{H}^{(A)},\,\widehat{H}^{(A)} = (\widehat{A}^{-1})^H \widehat{H} \widehat{A}^{-1},
\end{align}
respectively. It can be easily verified that
\begin{subequations}
\begin{enumerate}
  \item
  \begin{align}\label{c1}
   \widehat{\mathbb{R}}_\leq\cap\widehat{\mathbb{P}}=-\widehat{\mathbb{R}}_{\mathrm{d},\leq}\cap\widehat{\mathbb{P}}_d,
   \, \widehat{\mathbb{R}}_\geq\cap\widehat{\mathbb{E}}=-\widehat{\mathbb{R}}_{\mathrm{d},\geq}\cap\widehat{\mathbb{E}}_d
  \end{align}
 where $\widehat{\mathbb{R}}_{\mathrm{d},\leq}:=\{X\in \mathrm{dom}({\mathcal{\widehat{R}}_d})\, |\, X  \leq  \widehat{\mathcal{R}}_d(X) \}$,
 \vspace{1mm}
  $\widehat{\mathbb{R}}_{\mathrm{d},\geq}:=\{X\in \mathrm{dom}({\mathcal{\widehat{R}}_d})\, |\, X  \geq  \widehat{\mathcal{R}}_d(X) \}$, $\widehat{\mathbb{P}}_d:=\{X\in \mathrm{dom}(\widehat{\mathcal{R}}_d) \,|\, \widehat{R}_d+\widehat{B}_d^H X\widehat{B}_d>0\}$ and $\widehat{\mathbb{E}}_d:=\{X\in \mathrm{dom}(\widehat{\mathcal{R}}_d) \,|\, \widehat{R}_d+\widehat{B}_d^H X\widehat{B}_d<0\}$.
  \item If $\widehat{A}-\widehat{B}\widehat{F}$ is invertible, then the matrix
  $\widehat{A}_d-\widehat{B}_d\widehat{F}_d$ is invertible with $\widehat{F}_d=-\widehat{R}^{-1}(\widehat{B}^H \widehat{H}+\widehat{R}\widehat{F})(\widehat{A}-\widehat{B}\widehat{F})^{-1}$ and satisfies
  \begin{align}\label{c2}
  ({A}_d-{B}_d{F}_d)(\widehat{A}_d-\widehat{B}_d\widehat{F}_d)=I.
  \end{align}
  That is, the matrix pair $(\widehat{A},\widehat{B})$ is anti-stabilizable if the matrix pair $(\widehat{A}_d,\widehat{B}_d)$ is stabilizable.
\end{enumerate}
\end{subequations}
The matrix pair $(\widehat{A}_d,\widehat{B}_d)$ is stabilizable is equivalent to $\widehat{\mathbb{T}}_d\cap\widehat{\mathbb{P}}_d\neq\emptyset$ when $\widehat{\mathbb{R}}_{\mathrm{d},\leq}\cap\widehat{\mathbb{P}}_d\neq\emptyset$ or $\widehat{\mathbb{R}}_{\mathrm{d},\geq}\cap\widehat{\mathbb{E}}_d\neq\emptyset$. Analogously, the matrix pair $(\widehat{A},\widehat{B})$ is anti-stabilizable is equivalent to $\widehat{\mathbb{O}}\cap\widehat{\mathbb{E}}\neq\emptyset$ when $\widehat{\mathbb{R}}_\leq\cap\widehat{\mathbb{P}}\neq\emptyset$ or $\widehat{\mathbb{R}}_\geq\cap\widehat{\mathbb{E}}\neq\emptyset$.

\end{enumerate}

\end{proof}
 Let 
$G_{{\lambda}}(C):=\bigcup\limits_{1\leq k \leq n}\textrm{Ker}(C-\lambda I)^k$ be the eigenspace and generalized eigenspace of a $n$-square matrix $C$ corresponding to the eigenvalue ${\lambda}\in\sigma(C)$, respectively.
Given a Hermitian matrix $Q$, the subset $U_Q(C)$ of the spectrum of a matrix $C\in\mathbb{C}^{n\times n}$ is defined by
\begin{subequations}\label{defU}
\begin{align}\label{defU1}
U_Q(C):=\{\lambda\in\sigma({C});\textrm{Ker}(C^H-\bar{\lambda} I)\cap \mbox{Ker}({Q})=\{0\}\}
\end{align}
or equivalently,
\begin{align}\label{defU2}
U_Q(C)\equiv\{\lambda\in\sigma(C);\mbox{rank}(\bb C-\lambda I & {Q}\eb)=n\},
\end{align}
where $Q=Q^H\in\mathbb{H}_n$. 
\end{subequations}
Note that $\mathbb{D}_\geq\cap\sigma(C)\subseteq U_Q(C)$ is equivalent to $(C,{Q})$ is stabilizable~\cite[Theorem~4.5.6]{l.r95}. Since the proof of main result is a little lengthy we shall divide it in several parts presented as a proposition and four auxiliary Lemmas.
Now, the following result can be readily obtained, which will play an important role in the section.
\begin{Lemma}\label{lem1111}

Let $C\in\mathbb{C}^{n\times n}$ and $Q\geq 0$.
Assume that there exist $X,Y\in\mathbb{H}_n$ such that
\begin{align}\label{13}
 \mathcal{S}_C(X)\geq (Y C)^HQ(Y C).
\end{align}
Suppose that $\sigma(C)\cap \mathbb{D}_=\subseteq U_Q(C)$. Then
   $X\geq 0$ 
      if $\rho(C)\leq 1$, and $X\leq 0$ 
      if $\mu(C)\geq 1$.
\end{Lemma}
\begin{proof}
Our attention is focused on the case in which $\rho(C)\leq 1$
since the other proof procedure is very similar. We first claim that
\begin{align}\label{15}
\mbox{G}_{\mathbb{D}_{=}\cap U_{Q}(C)}(C)\cap\mbox{Ker}(X-Y)\subseteq\mbox{Ker}(X)\cap\mbox{Ker}(Y).
 \end{align}
For the sake of convenience, we factorize $Q$ as the Cholesky decomposition  $Q=Q_1 Q_1^H$. Note that $\mbox{Ker}(Q)=\mbox{Ker}(Q_1^H)$.
Let ${\lambda}\in\sigma(C)\cap\mathbb{D}_=$ and $u\in\mbox{Ker}(C-\lambda I)$ with $u\neq 0$. We notice that $\bar{\lambda}\in\sigma(C^H)$ and $u^H C^H=\bar{\lambda}u^H$. Then, a direct computation $0=u^H \mathcal{S}_{C}(X)u\geq (u^H C^HY{Q_1})(u^H C^HY{Q_1})^H\geq 0$ implies that
\begin{align*}
  0&=u^H C^H {Y}{Q_1}=\bar{\lambda}u^H{Y}{Q_1}, \\
  0&=u^H \mathcal{S}_{C}(X)=\bar{\lambda}u^H X( \lambda I-C).
\end{align*}
Under the assumption we conclude that $Xu=Yu=0$ since $u\in\mbox{Ker}(X-Y)$.
Let $(C-\lambda I)v=u$ and thus $(C-\lambda I)^2v=0$. Observing the quadratic form
$v^H \mathcal{S}_{C}(X)v$ again gives us
\begin{align*}
0&= v^H(X-C^HXC)v=v^HXv-(\overline{\lambda} v^H+u^H)X(\lambda v+u)\\
&=-(\lambda u^H X v+\bar{\lambda}v^H X u+u^H X u)\geq (v^H C^HY{Q_1})(v^H C^HY{Q_1})^H\geq 0,
\end{align*}
and we have some analog equalities
\begin{align*}
  0&=v^H C^HY{ Q_1} =\bar{\lambda}v^H Y{Q_1}, \\
  0&=v^H \mathcal{S}_{C}(X)=\bar{\lambda}v^H X(\lambda I-C).
\end{align*}
Therefore, $Xv=Yv=0$ for $v\in\mbox{Ker}(X-Y)$ by using the assumption, and an induction argument proves that $X v_i=Y v_i=0$ for a Jordan chain $\{v_i\}$ corresponding to the unimodular eigenvalue $\lambda$ of $C$. Here $\{v_i\}\subseteq\mbox{Ker}(X-Y)$. This completes the proof of~\eqref{15}.

{Let $J_{C}=P^{-1}C P$ be the Jordan canonical form of $C$. Suppose that $J_{C}= J_1 \oplus J_s$, where $J_1\in\mathbb{C}^{m_1\times m_1}$ with $\sigma(J_1)\subseteq\mathbb{D}_=$, $J_s\in\mathbb{C}^{m_2\times m_2}$ with $\rho(J_s)< 1$}. Let $\widehat{X}:=P^H X P$.
Note that
$
\mathcal{S}_{J_C}(\widehat{X})=\widehat{X}-J_{C}^H\widehat{X} J_{C}=P^H\mathcal{S}_{C}({X})P.
$
Concerning the part~1, a trivial verification shows that $\widehat{X}= \widehat{X}_{s} \oplus 0_{m_2}$ when each assumption hold, where $\widehat{X}_{s}\in\mathbb{H}_{m_1}$. Further computation having $\mathcal{S}_{J_s}(\widehat{X}_s)\oplus 0_{m_2}=\mathcal{S}_{J_C}(\widehat{X})\geq \widehat{Y}:=P^H((Y C)^HQ(Y C))P\geq 0$ yield $\widehat{Y}= \widehat{Y}_{s} \oplus 0_{m_2}$, where $\widehat{Y}_{s}\in\mathbb{H}_{m_1}$. Observes that $\widehat{X}_{s}\in\mathbb{H}_{m_1}$ satisfies the inequality $\mathcal{S}_{J_s}(\widehat{X}_{s})\geq\widehat{Y}_{s}\geq 0$. Therefore, $X=P^{-H}(\widehat{X}_{s}\oplus 0_{m_2})P^{-1}\geq 0$ since
$\widehat{X}_{s}\geq\mathcal{S}_{J_s}^{-1}(\widehat{Y}_{s})$ is a positive semidefinite matrix when $\rho(J_s)<1$. The part~2 is very similar and thus the proof is omitted.

%
\end{proof}

An useful property about \eqref{Req-c} is characterized in the following result.
\begin{Proposition}\label{2222}
Assume that $H\in\mathrm{dom}({\mathcal{R}})$. Let $\Delta_{Y,X}:=Y-X$ for any $X,Y\in\mathbb{H}_n$.
\begin{enumerate}
  \item For any $X,Y\in\mathrm{dom}(\mathcal{R})$, we have
    \begin{align*}
  \widehat{K}(Y,X)=({\Delta_{Y,X}}\widehat{T}_{Y})^H[\widehat{B}\widehat{R}^{-1}_X\widehat{B}^H]({\Delta_{Y,X}}\widehat{T}_{{Y}}).
   \end{align*}
  \item For any $X\in\mathbb{R}_\leq$ and $Y\in \mathbb{R}_\geq$, then $\mathcal{S}_{\widehat{T}_{Y}}(\Delta_{Y,X})\geq \widehat{K}(Y,X).$
\end{enumerate}
\end{Proposition}
\begin{proof}\par\noindent

\begin{enumerate}
\item[(1)]
   From $\widehat{F}_{Y}-\widehat{F}_{X}=\widehat{R}_Y^{-1}\widehat{B}^H({Y}-{X})\widehat{T}_{{X}}
      =\widehat{R}_X^{-1}\widehat{B}^H({Y}-{X})\widehat{T}_{{Y}}$ we deduce that
\[
\widehat{K}(Y,X)=\widehat{T}_{Y}^H({Y}-{X})\widehat{B}\widehat{R}_X^{-1}\widehat{B}^H({Y}-{X})\widehat{T}_{Y},
\]
and the resulting equality immediately follows.
\item[(2)]
The proof is based on the following observation. From the
assumption and \eqref{Req-c} we have
\begin{align*}
 \mathcal{S}_{\widehat{T}_{Y}}(\Delta_{Y,X})
 &=\mathcal{S}_{\widehat{T}_{Y}}({Y})-\mathcal{S}_{\widehat{T}_{Y}}(X)
          \geq\widehat{H}+\widehat{K}(Y,0)-\left[ \widehat{H}+\widehat{K}(Y,0)-\widehat{K}(Y,X) \right] = \widehat{K}(Y,X).
\end{align*}
This completes the proof.

\end{enumerate}
\end{proof}
With Proposition~\ref{uniq}, we claim that almost stabilizing and
almost anti-stabilizing solutions coincide with some extremal solutions by using several lemmas. The first result gives more insights into the feature of the upper and lower bounds of $\widehat{\mathbb{R}}_{\leq}\cap\widehat{\mathbb{P}}$ and $\widehat{\mathbb{R}}_{\geq}\cap\widehat{\mathbb{E}}$, which is the key for the approach presented.
\begin{Lemma} \label{uniq}
\par\noindent
Assume that $Y\in\widehat{\mathbb{R}}_\geq$ with property
 \begin{align}\label{P}
\mbox{rank}(\bb \widehat{T}_Y-\lambda I & \widehat{B}\eb)=n\quad\mbox{for all }\lambda\in\sigma(\widehat{T}_Y)\cap \mathbb{D}_=.
 \end{align}
 Let $X\in\widehat{\mathbb{R}}_{\leq}\cap\widehat{\mathbb{P}}$ and $Z\in\widehat{\mathbb{R}}_{\leq}\cap\widehat{\mathbb{E}}$.
Then,
\begin{enumerate}
  \item $Z\geq Y \geq X$ if $\rho(\widehat{T}_Y)\leq 1$.
  \item $Z\leq Y \leq X$ if $\mu(\widehat{T}_Y)\geq 1$.
\end{enumerate}

\end{Lemma}
\begin{proof}
We only need to prove the part~1, since the proof procedure of the remaining parts are very similar.
It can be verified by Proposition~\ref{2222} that
\begin{align*}
\mathcal{S}_{\widehat{T}_{Y}}(\Delta_{Y,X})&\geq \widehat{K}(Y,X)=({\Delta_{X,Y}}\widehat{T}_{X})^H[\widehat{B}\,\widehat{R}_{X}^{-1}\widehat{B}^H](\overline{\Delta}_{X,Y}\widehat{T}_{{X}})\geq 0.
\end{align*}
In addition, we observe that
$\sigma(\widehat{T}_X)\cap \mathbb{D}_=\subseteq U_{\widehat{B}\,\widehat{R}_{X}^{-1}\widehat{B}^H}(\widehat{T}_X)$
by the assumption~\eqref{P}. Note that $\mathrm{Ker}(\widehat{B}\,\widehat{R}_{X}^{-1}\widehat{B}^H)=\mathrm{Ker}(\widehat{B}^H)$.
Therefore by applying Lemma~\ref{lem1111} we see that $Y-X\geq 0$ and the remainder inequality $Y-Z\leq 0$ can be proved in the same way.

%
%
%
%
%
%
\end{proof}
 The next lemma shows that the property~\eqref{P} is a consequence of the assumption $\mathbb{T}\neq\emptyset$ or $\mathbb{O}\neq\emptyset$.
\begin{Lemma}\label{lem3p3}
Assume that $H\in\mathrm{dom}(\mathcal{R})$. Then, 
the property~\eqref{P} holds if $\widehat{\mathbb{T}}\neq\emptyset$ or $\widehat{\mathbb{O}}\neq\emptyset$.
\end{Lemma}
\begin{proof}
The proof is straightforward by the observation
\[
\mbox{rank}(\bb \widehat{T}_Y-\lambda I & \widehat{B}\eb)=\mbox{rank}(\bb \widehat{A}-\widehat{B}\widehat{F}_Y-\lambda I & \widehat{B}\eb)=\mbox{rank}(\bb \widehat{A}-\lambda I & \widehat{B}\eb),
\]
and the pair $(\widehat{A},\widehat{B})$ is stabilizable (resp. anti-stabilizable) if $\widehat{\mathbb{T}}\neq\emptyset$ (resp. $\widehat{\mathbb{O}}\neq\emptyset$) by using Lemma~\ref{thm3p34}
.
\end{proof}
As a consequence of the above two lemmas, the following result can also be easily obtained.
\begin{Lemma} \label{uniq2}
\par\noindent
\begin{enumerate}
\item
Assume that the matrix pair $(\widehat{A},\widehat{B})$ is stabilizable. When $\widehat{\mathbb{R}}_\leq\cap\widehat{\mathbb{P}}=\emptyset$ (resp. $\widehat{\mathbb{R}}_\geq\cap\widehat{\mathbb{E}}=\emptyset$), $\mathbf{\widehat{X}_{a.s.}}$ exists and unique. In addition, $\mathbf{\widehat{X}_{a.s.}}=\mathbf{\widehat{X}_M}$ (resp. $\mathbf{\widehat{X}_{a.s.}}=\mathbf{\widehat{Y}_m}$).
\item
Assume that the matrix pair $(\widehat{A},\widehat{B})$ is anti-stabilizable. When $\widehat{\mathbb{R}}_\leq\cap\widehat{\mathbb{P}}=\emptyset$ (resp. $\widehat{\mathbb{R}}_\geq\cap\widehat{\mathbb{E}}=\emptyset$), $\mathbf{\widehat{X}_{a.a.s.}}$ exists and unique. In addition, $\mathbf{\widehat{X}_{a.a.s.}}=\mathbf{\widehat{X}_m}$ (resp. $\mathbf{\widehat{X}_{a.a.s.}}=\mathbf{\widehat{Y}_M}$).
\end{enumerate}

\end{Lemma}
\begin{proof}
Only the proof of Item 2 are given. The Item 1 can be easily proven by applying the same argument The existence of an almost anti-stabilizing solution $\mathbf{{X}_{a.a.s.}}=\mathbf{{X}_{m}}$ is guaranteed by Lemma~\ref{thm3p34} since $\widehat{\mathbb{R}}_\leq\cap\widehat{\mathbb{P}}=\emptyset$ and $\widehat{\mathbb{T}}\neq\emptyset$.
The result is a consequence of Lemma~\ref{uniq}.
\end{proof}
Lemma~\ref{uniq2} together with Proposition~\ref{pro33} implies that following main theorem. We state the fourth main results but without proof.
\begin{Theorem} \label{uniq3}
\par\noindent
\begin{enumerate}
  \item Assume that the matrix pair $\mathbb{T}\neq\emptyset$.
\begin{enumerate}
\item
When ${\mathbb{R}}_\leq\cap{\mathbb{P}}\neq\emptyset$, $\mathbf{{X}_{a.s.}}$ exists and $\mathbf{{X}_{a.s.}}=\mathbf{{X}_M}$.
  \item
When ${\mathbb{R}}_\geq\cap{\mathbb{E}}\neq\emptyset$, $\mathbf{{X}_{a.s.}}$ exists and
$\mathbf{{X}_{a.s.}}=\mathbf{{Y}_m}$.
\end{enumerate}
\item
Assume that $\mathbb{O\neq\emptyset}$.
\begin{enumerate}
  \item When ${\mathbb{R}}_\leq\cap{\mathbb{P}}\neq\emptyset$, $\mathbf{{X}_{a.a.s.}}$ exists and $\mathbf{{X}_{a.a.s.}}=\mathbf{{X}_m}$.
  \item When ${\mathbb{R}}_\geq\cap{\mathbb{E}}\neq\emptyset$, $\mathbf{{X}_{a.a.s.}}$ exists and
$\mathbf{{X}_{a.a.s.}}=\mathbf{{Y}_M}$.
\end{enumerate}
\end{enumerate}
\end{Theorem}

\section{Concluding remarks} \label{conclusion}
In this paper we investigate eight extremal solutions and almost (anti-)stabilizing solution of a class of conjugate discrete-time Riccati equations, arising originally from the LQR control problem for discrete-time antilinear systems. Moreover, it is proved that the existence of $(\mathbf{X_M},\mathbf{X_m})$ and $(\mathbf{Y_M},\mathbf{Y_m})$ are mutually exclusive. We believe that our theoretical results would be useful in the LQR control problem, or even the state-feedback stabilization problem, for discrete-time antilinear systems. It is not only to be studied that the existence of eight extremal solutions of the CDARE \eqref{cdare} to be investigated in this work, and
{ it also leads to our future work that how to apply the accelerated techniques presented in the series works \cite{l.c18,l.c20}
for computing eight extremal solutions to the CDARE \eqref{cdare}, simultaneously. The dual-CDARE~\eqref{dual} with singular matrix $A$ also will be studied for published elsewhere.}
\section*{Acknowledgment}
This research work is partially supported by the National Science and Technology Council and the National Center for Theoretical Sciences in Taiwan.
The author would like to thank the support from the National Science and Technology Council of Taiwan under the grant MOST 113-2115-M-150-001-MY2.
\section*{Declaration of competing interest}
The author declare that he has no known competing financial interests or personal relationships that could have appeared
to influence the work reported in this paper.
\bibliographystyle{abbrv}

\end{document}